\documentclass[12pt,reqno,a4paper]{amsart}
\usepackage{euscript,fullpage,amsmath,amssymb}
\usepackage{enumerate}
\usepackage{graphicx}
\usepackage{xcolor}

\newtheorem{thm}{Theorem}[section]
\newtheorem{lemma}[thm]{Lemma}
\newtheorem{prop}[thm]{Proposition}
\newtheorem{cor}[thm]{Corollary}

\newtheorem{example}[thm]{Example}
\newtheorem{rmk}[thm]{Remark}
\newtheorem{com}[thm]{Comment}

\newbox\ncintdbox \newbox\ncinttbox
\setbox0=\hbox{$-$} \setbox2=\hbox{$\displaystyle\int$}
\setbox\ncintdbox=\hbox{\rlap{\hbox
    to \wd2{\kern-.1em\box2\relax\hfil}}\box0\kern.1em}
\setbox0=\hbox{$\vcenter{\hrule width 4pt}$}
\setbox2=\hbox{$\textstyle\int$}
\setbox\ncinttbox=\hbox{\rlap{\hbox
    to \wd2{\kern-.14em\box2\relax\hfil}}\box0\kern.1em}

\let\phi\varphi

\let\epsilon\varepsilon

\let\tilde\widetilde

\newcommand{\Z}{{\mathbb Z}}
\newcommand{\N}{{\mathbb N}}
\newcommand{\T}{{\mathcal{T}}}

\def\Le{\text{Leb}}
\def\a{\alpha}

\def\eps{\varepsilon}

\def\o{\omega}

\def\ox{\overline{x}}
\def\oz{\overline{z}}
\def\B
\def\di{\setminus}

\def\N{\ensuremath{\mathbb N}}

\def\Z{\ensuremath{\mathbb Z}}

\def\B{\ensuremath{\mathcal B}}

\def\P{\hat{P}}

\def\o{\ensuremath{\underline{\omega}}}

\def\eps {\varepsilon}

\def\T{\hat{T}}

\def\oom{\overline{\omega}}

\title{Extreme Value Theory for synchronization of  Coupled Map Lattices}
 \date{\today}
 
\begin{document}
 \begin{abstract}
We show that the probability of the appearance of synchronization in chaotic coupled  map lattices is related to the distribution of the maximum of a certain observable evaluated along almost all orbits.  We show that such a distribution belongs to the family of  extreme value laws, whose parameters, namely the extremal index,  allow us  to get a detailed description of the probability of synchronization. Theoretical results are supported by robust numerical computations that allow us to go beyond the theoretical framework provided and are potentially applicable to physically relevant systems.
\end{abstract}
 \maketitle
 \begin{center}
\authors{D. \ Faranda \footnote{LSCE-IPSL, CEA Saclay l'Orme des Merisiers, CNRS UMR 8212 CEA-CNRS-UVSQ,
Universit\'e Paris-Saclay, 91191 Gif-sur-Yvette, France.} \footnote{London Mathematical Laboratory, London, UK
. E-mail: {\tt \email{davide.faranda@lsce.ipsl.fr}}.},
H.\ Ghoudi\footnote{ Laboratory of Dynamical Systems and Combinatorics, Sfax University Tunisia and  Aix Marseille Universit\'e, CNRS, CPT, UMR 7332, 13288 Marseille, France and Universit\'e
de Toulon, CNRS, CPT, UMR 7332, 83957 La Garde. Email: {\tt \email{ghoudi.hamza@gmail.com }}.},
P.\   Guiraud\footnote{CIMFAV, Facultad de Ingeniería, Universidad de Valparaíso, Valpara\`iso, Chile. E-mail: {\tt \email{pierre.guiraud@uv.cl}}.},
S.\ Vaienti\footnote{Aix Marseille Univ, Universit\'e de Toulon, CNRS, CPT, Marseille, France. E-mail: {\tt \email{vaienti@cpt.univ-mrs.fr}}.}}

\end{center}

\section{Introduction}
Coupled Map Lattices (CML)  are discrete time and space dynamical systems introduced in the mid 1980's by Kaneko and Kapral
as suitable models for the study and the numerical simulation of nonlinear phenomena in spatially extended systems.
The phase space of a CML is a set of  scalar (or vector)  sequences indexed by a lattice $L$,  e.g. $L=\Z^d,\ \Z$ or $\Z/n\Z$.
For instance, a configuration ${\ox\in I^L}$ of the lattice may represent a spacial sample of a  mesoscopic quantity with value in an interval $I$, such as a chemical concentration, the velocity of a fluid, a population density or a magnetization.
The dynamics of  the lattice is given by a map $\T:I^{L}\to I^{L}$ which is usually written as the composition of two maps, i.e. $\T:=\Phi_\gamma\circ\T_0$, where $\T_0:I^{L}\to I^{L}$ is called the {\it uncoupled dynamics} and $\Phi_\gamma :I^{L}\to I^{L}$ the {\it coupling operator}.
The uncoupled dynamics acts on a configuration $x\in I^{L}$ as the product dynamics of
a {\it local map} $T:I\to I$, that is $\T_0(x)_i:=T(x_i)$ for every $i\in L$.  The coupling operator models  spacial interactions, which intensity is given by the parameter $\gamma\in[0,1]$.
In particular, in the absence of interaction  $\gamma=0$ and $\Phi_0=Id$.
For example, for $L=\Z$ or $L=\Z/n\Z$ the coupling operator often writes as
\begin{equation}\label{COUPLINGINT}
(\Phi_\gamma(x))_i:=\sum_{j\in L}c_{\gamma,j}x_{i-j} \quad\forall i\in L,
\end{equation}
where $c_{\gamma,j}\geq 0$, $\sum_{j\in L}c_{\gamma,j}=1$ and $c_{0,0}=1$.

In the huge literature about CML, one can find many possible choices for the local map and the coupling operator. For instance, the dynamics of CML of bistable, unimodal, or chaotic maps have been studied for different kind and range of coupling, revealing a rich phenomenology including spatial chaos, stable periodic points, space-time chaos, clusters, traveling waves and synchronization, (see \cite{CK87, FC, K93} and references therein).
In this paper, we will consider a system of $n$ coupled chaotic local maps (the precise properties are given in Section \ref{MAPOP})  defined  for
any $\ox:=(x_1,\dots,x_n)\in I^n$ by:
$$
(\T(\ox))_i=(1-\gamma)T(x_i)+\frac{\gamma}{n}\sum_{j=1}^n T(x_{j})\qquad \forall i\in \{1,\dots,n\}.
$$
Note that the study of this system is equivalent to that of a CML  on a periodic lattice where the coupling operator is defined by \eqref{COUPLINGINT} with $L=\Z/n\Z$, $c_{\gamma,0}=(1-\gamma+\frac{\gamma}{n})$ and  $c_{\gamma,j}=\frac{\gamma}{n}$ for all $j\in\{1,\dots,n-1\}$.
The chaotic and synchronization properties of CML of logistic local maps with this mean-field-type global coupling were observed and studied by Kaneko in \cite{KA} and then, among others,  by P. Ashwin \cite{PA} (and references therein).


%

The first contribution which looked at CML in the framework and with the tools of ergodic theory,
was the work by Bunimovich and Sinai. In the famous paper \cite{BS}, using thermodynamic formalism, they proved the existence of mixing SRB measures for infinite CML with chaotic local map and weak (nearest
neighbor) coupling.  Since then, the progress in the study of the  statistical properties of chaotic CML has been enormous, with the contribution of several people, and the development
 of a spectral theory \cite{KELI, KELIV}. We defer to the book \cite{FC} for a wide panorama on the different approaches to CML and for exhaustive references.

In this paper, we present a new application of Extreme Value Theory (EVT) to CML on a finite (or periodic) lattice. Our aim  is to provide a first approach to CML by using EVT and to show how to get a certain number of rigorous results about the statistics of some rare events, such as the synchronization in chaotic CML. We say that the CML is synchronized when it is near a homogeneous configuration (in a small neighborhood of the diagonal of the phase space).   Synchronization is usually intended to last for a while once it has started and this is what usually happens for some kinds of  chains of synchronized oscillators. This is not the case  of course for chaotic CML, since almost every orbit is recurrent by the Poincar\'e Theorem. What we actually investigate is therefore  the {\em probability of a first synchronization and how long we should wait to get it with a prescribed accuracy}. EVT provides this kind of quantitative information, since  synchronization processes can be interpreted and quantified by computing the asymptotic distribution of the maximum of a suitable random process, see Sections \ref{EVTL} and \ref{EVTS}.

Although we could not get a global synchronization persisting in time, we could ask about the distribution of the number of successive synchronization events when the systems evolves up to a certain time. We will see that after a suitable rescaling, the distribution of that number follows a compound Poisson statistics: it is worth mentioning that for two {\em uncoupled} expanding maps of the circle, this result dates back to a paper by Coelho and Collet, \cite{CC}. \\

Actually a first result in our direction was given in the paper \cite{CG}, although not explicitly related to EVT, where the authors considered two coupled interval maps and applied their spectral theory of   open systems with holes to investigate the first entrance of the two components into a small strip along the diagonal, which is equivalent to the synchronization of the two-components lattice up to a certain accuracy.
 In more general situations, we will present arguments
about the spectral properties of the transfer  operator of the system
to sustain the existence of a limit distribution for  the maxima of some observables related to synchronization, and we will discuss a formula approximating  the  {\em extremal index} (a parameter of the distribution)  for lattices with an arbitrary  number of  components.   We therefore estimate the  behavior of such an index  when the number of components is large. We will then generalize the theory to  CML which are  randomly perturbed with additive noise and show, in particular with numerical evidence,  that the extremal index is $1$ for any dimension of the lattice. We hope that our approach could be helpful to understand and quantify those  phenomena, like in neuronal spikes or in business cycles of financial markets, where bursts  of synchronization happen, disappear, happen again, apparently in a disordered manner, but very often following the extreme distributions arising in chaotic systems.\\

In Section 2, we present a powerful and general approach based on  perturbation of the transfer operator, and which has the advantage of being applicable to a large class of observables arising in the study of EVT. In Section 3, we give a short insight into basic notions of EVT, especially when it is applied to recurrence in dynamical systems. In particular, we define the extremal index and show that it goes to one when the size of the lattice goes to infinity or in presence of noise. In Section 4, we apply EVT to compute the probability of synchronization events, and sustain the results by computing  the extremal index in Section 5. This computation depends on the behavior of the invariant density in the neighborhood of the diagonal; our formula (\ref{EEIII}) can be proved under the assumption {\bf P8} which we believe to be unavoidable.  In Section 6, we study the distribution of the number of successive synchronization events. In Section 7, we  show that our analytic results and estimates  are supported  by numerical computations. They confirm the existence of an extreme value distribution for a different kind of synchronization, which we called {\em local}, and they validate the expected compound Poisson statistics for the distribution of the number of successive  visits.  The fact that the extremal index for local synchronization seems not to depend on the size of the lattice   is an interesting numerical discovery. In forthcoming papers we will study more general CML with non-local form of coupling including the important case of diffusive or Laplacian interaction. A few other possible developments are presented at the end of the paper (see section 7.2).

\section{The map and the operators}\label{MAPOP}
As mentioned in the Introduction, we  consider a finite CML  of size $n\geq 2$ with a local map  $T: I\to I$ and a  global coupling. It is defined for any $\overline{x}=(x_1,\dots,x_n)\in I^n$ and  $\gamma\in[0,1]$ by \footnote{We will not index the map $\T$ with $n$, hoping it will be clear from the context.}:
\begin{equation}\label{TM}
\T(\overline{x})_i=(1-\gamma)T(x_i)+\frac{\gamma}{n}\sum_{j=1}^nT(x_j)\qquad \forall\, i\in\{1,2,\dots,n\},
\end{equation}
where $\overline{x}=(x_1,\dots,x_n)\in I^n, \gamma\in[0,1]$. We suppose that $T$ is a piece-wise  expanding map of the unit interval onto itself,  with a finite number of branches, say $q,$ and which we take of class $C^2$ on the interiors of the domains of injectivity $A_1,\dots, A_q,$ and extended by continuity to the boundaries. The $C^2$ assumption is used in the proof of Propositions (\ref{propp}) and (\ref{proppp}), although it could be relaxed with a $C^{1+\alpha}$ condition. Instead the finitness of the number of branches is
widely used in almost all the arguments. Let us denote by $U_k, k=1,\dots, q^n,$ the domains of local injectivity of $\hat{T}$.  By the previous assumptions on $T$, there exist open sets $W_k\supset U_k$ such that $\T_{|W_k}$ is a $C^2$ diffeomorphism (on the image). We will require that
$$
s_n:=\sup_k\sup_{\ox\in \T (W_k)}||D\T_{|W_k}^{-1}(\ox)||< \lambda<1,
$$
where
$\lambda:=\sup_i\sup_{x\in T(A_i)}|DT_{|A_i}^{-1}(x)|,$ and $||\cdot||$ stands for the euclidean norm. We will write {\em dist} for the distance with respect to this norm.

 An important tool  for our further considerations is  the transfer, or Perron-Frobenius (PF), operator. The PF operator $\P$ of the map $\T$ is simply defined by the duality integral relation
  $$
  \int \hat{P}(f) g d\Le=\int f g\circ \hat{T} d\Le,
  $$
  where $\Le$ denotes the Lebesgue measure on $I^n,$ $f\in L^1$ and $g\in L^{\infty}.$\footnote{In the following we will use the same symbol $\Le$ for any $n$. Moreover $L^1,$ $L^p$ and $L^{\infty}$ will be taken with respect to $\Le.$ Finally the integral with respect to Lebesgue measure will be denoted with $\int d\Le(x)$ or $\int dx.$} The spectral properties of the PF operator become interesting when it acts on suitable Banach spaces. Let us therefore suppose that there exists a Banach space $\B$ with norm $||\cdot||_{\B},$ which is compactly injected in $L^1$  and the following properties hold\footnote{ We will call a Banach space with this property {\em adapted} (to $L^1$).}:
  \begin{itemize}
  \item {\bf P1} (Lasota-Yorke inequality) For any $f\in \B$ there exists $\eta<1$ and $C>0$ such that
  $$
  ||\hat{P}f||_{\B}\le \eta ||f||_{\B}+C||f||_1.
  $$
  \end{itemize}

\noindent The Lasota-Yorke inequality implies that $\hat{P}$ has an isolated eigenvalue equal to $1$  which is also the spectral radius of $\hat{P}$ ({\em spectral gap property}). We will often call $\eta$ the contraction factor in the Lasota-Yorke inequality.
  \begin{rmk}\label{LLYY}
  By iterating the previous inequality one easily get that
  \begin{equation}\label{ILY}
  ||\hat{P}^kf||_{\B}\le \eta^k ||f||_{\B}+\frac{C}{1-\eta}||f||_1,\quad \forall\, k>1.
  \end{equation}
  This last inequality is actually  needed in the perturbation theory used below. If one cannot achieve it because {\bf P1} fails, it is enough to get {\bf P1} for an iterate of $\hat{T}.$ In this case a standard argument allows us to get again (\ref{ILY}).
  \end{rmk}

\begin{itemize}
  \item {\bf P2} The eigenvalue $1$ is simple and $\hat{P}$ has no other eigenvalue on the unit circle.
This implies that $\hat{P}$ preserves a  mixing measure $\hat{\mu}$ which is the unique absolutely continuous invariant measure with respect to Lebesgue. We moreover assume that the associated density $\hat{h}\in L^{\infty}.$
  \end{itemize}

It is well known that with our assumptions on $T$, the uncoupled dynamics $\hat{T}_0$, i.e $\gamma=0$ in (\ref{TM}), satisfies {\bf P1} on any {\em reasonable} functional space $\B$. We will give examples of such spaces just below. Therefore, the spectral decomposition theorem of Ionescu-Tulcea-Marinescu, see for instance \cite{HH}, guarantees the existence of a finite number of absolutely continuous ergodic components.
They reduce to a unique absolutely continuous mixing measure, which is  {\bf P2}, with some topological transitivity condition on the map $T$, which could be achieved by asking, for instance,  $T$ to be Bernoulli, Markov, covering, etc (see, e.g.,  Example \ref{EX0}).

In order to transfer the properties {\bf P1} and {\bf P2} to the map $\hat{T}$ with $\gamma>0,$ we invoke the
   perturbation theory by Keller and Liverani developed in \cite{KL}. According to that theory, one should previously show the persistence of the Lasota-Yorke inequality (\ref{ILY}) for the map $\hat{T}$ and then   check that, for any $f\in \B$, we have
    \begin{equation}\label{TN}
    ||(\hat{P}-P_0)f||_1\le p_{\gamma}||f||_{\B},
     \end{equation}
     where $P_0$ is the PF operator of the uncoupled system ($\gamma=0$), and  $p_{\gamma}$ is a monotone upper semi-continuous function converging to $0$ when $\gamma$ goes to $0$. We defer again to Example \ref{EX0} for a particular case, where this technique can be applied.

  The aforementioned perturbation theory was successively improved in \cite{CG} by the same authors, in order to deal with {\em open} systems which   produce a different kind of perturbation for the transfer operator. This   perturbation  arises naturally in the context of the EVT, as we will see in the next section.
In order to introduce and define it, let $\{D_l\}_{l\in\N}$ be an increasing collection of  nested subsets of $I^n$ such that $\Le(D_l)\rightarrow 1$ when  $l\rightarrow \infty$. Moreover, suppose that the sets $D_l$ are the closures of their interiors and have piece-wise $C^{\infty}$ and co-dimension $1$ boundaries. According to the observable used for the application of EVT, the sets $D_l$ have a specific definition, and  they will be given by \eqref{DL1} and \eqref{set}. The EVT can be related to the spectral theory by considering the perturbed transfer  operator $\tilde{P}_l$, which is defined for any $h\in \B$ by:
$$
\tilde{P}_l(h):=\P(h {\bf 1}_{D_l}).
$$

We now add  new assumptions this operator must satisfy in order to apply the perturbation theory for open systems. The goal is to compare  the operators $\P$ and $\tilde{P}_l$ and  get an asymptotic expansion  for the spectral radius of $\tilde{P}_l$ close to $1$ for large values of $l$.  We will see that it  will give us the extremal index in the limiting distribution of  Gumbel's law. We follow in particular the scheme proposed by Keller in \cite{KE}, that we also summarized in \cite{AFV}, Section 5, and in Chapter 7 of the book \cite{ERDS} to which we defer for more details. There are 6 assumptions in \cite{KE}, Section 2. The first three ask for uniform (in the ``noise" parameter $l$) quasi-compactness for the operator $\tilde{P}_l$. We summarize them in the following single assumption:
\begin{itemize}
\item {\bf P3} The operators $\tilde{P}_l$ satisfy a Lasota-Yorke inequality, uniform in $l$,  on the space $\B,$ namely, the factors $\eta$ and $C$ are the same for every sufficiently large $l$.
\end{itemize}

\noindent The next two properties {\bf P4} and {\bf P5} cover assumptions (5) and (6) in Keller \cite{KE}. We also notice that {\bf P4}, together with {\bf P2}, implies assumption (4) in \cite{KE}, as explained in Remark 3 still in \cite{KE}.
\begin{itemize}
\item {\bf P4} For any $h\in \mathcal{B},$ the quantity
$$
r_l:=\sup_{h, ||h||_{\B}\le 1}|\int (\P h-\tilde{P}_lh)d\Le|
$$
goes to zero when $l\rightarrow  \infty$.
\item  {\bf P5} The density $\hat{h}$ of the (unique mixing) invariant measure $\hat{\mu}$ of $T$ verifies
    \begin{equation}\label{P5}
    r_l ||(\P-\tilde{P}_l)\hat{h}||_{\mathcal{B}}\le C' \hat{\mu}(D_l^c),
    \end{equation}
     where $C'$ is a constant independent of $l$ and $D_l^c$ denotes the complement of $D_l$. We  moreover assume that the density $\hat{h}$ is strictly positive, namely its infimum is larger than $\hat{h}^{( \mathrm{inf})}>0$ on a set of full measure.
\end{itemize}

\noindent We finally assume that

\begin{itemize}
\item {\bf P6} The following limit
 \begin{equation}\label{Q3}
 q_k:=\lim_{l\rightarrow \infty}q_{k,l}:=\lim_{l\rightarrow \infty}\frac{\int (\P-\tilde{P}_l)\tilde{P}_l^k(\P-\tilde{P}_l)(\hat{h})d\Le}{\hat{\mu}(D^c_l)}
 \end{equation}
exists for any $k\in \mathbb{N}\cup \{0\}$.
\end{itemize}

Under the assumptions {\bf P1}-{\bf P6}, it has been proved in \cite{CG} that
  \begin{equation}\label{Q2}
 \theta:=1-\sum_{k=0}^{\infty}q_k,
 \end{equation}
exists and  is equal to $\lim_{l\rightarrow \infty}\frac{1-\rho_l}{\hat{\mu}(D^c_l)}$, where $\rho_l$ is the spectral radius of $\tilde{P}_l$. Therefore we have the following asymptotic expansion for $\rho_l$:
 \begin{equation}\label{Q1}
 1-\rho_l=\hat{\mu}(D^c_l) \theta (1+o(1)), \ \text{in the limit} \ l\rightarrow \infty.
 \end{equation}
We stress that $\rho_l$ is the largest eigenvalue of $\tilde{P}_l,$ that there are no other eigenvalues on the  circle of radius $\rho_l,$ and that there exist functions $\hat{g}_l\in \B$ and measures $\hat{\mu}_l$  for which the operators $\tilde{P}_l$ satisfy
 \begin{equation}\label{deco}
 \tilde{P}_lh=\rho_l \hat{g}_l\int h d\hat{\mu}_l+Q_lh
 \end{equation}
for all $h\in \B$.
 Moreover  $\int \hat{g}_l d\hat{\mu}=1 ,$ $\int h d\hat{\mu}_l\rightarrow \int h d\hat{\mu}$ when $l\rightarrow \infty$ and finally $Q_l$ is a linear operator with spectral radius strictly less than $\rho_l$ and satisfying: $||Q_l^n||_{\B}\le \varsigma_{l}^n$, for a suitable $0<\varsigma_{l}<1$, see again \cite{CG} for the derivation of these formulas.\\

It is a remarkable fact that this approach automatically provides the  scaling exponent $\theta$  for the asymptotic distribution of the maxima, see (\ref{PIV}) below, and therefore it gives a  new proof of the existence of that distribution.  The quantity $\theta$ is called the extremal index (EI) and it will play an important role in the following. We will see in particular that it  gives a correction to the pure exponential law for the distribution of the maxima. In that respect it coincides with the extremal index as it is defined in EVT, see \cite{AJEM2}, \cite{ERDS}. Our next task will therefore be to look for a Banach space which verifies the preceding six  properties.

One natural candidate  would be the {\em space $BV(I^n)$  of  functions of bounded variation on $\mathbb{R}^n$ restricted to the $L^1$ functions supported on $\tilde{I}^n:=\text{interior}(I^n).$}
 This space was used in \cite{CG} in dimension $2$, but it seems difficult to use it in higher dimensions to obtain {\bf P5}. The reason is that in order to get {\bf P5} one needs first to compute the quantity $r_l$ in {\bf P4}. Since $h$ may not be necessarily in $L^{\infty}$, we should use Sobolev's inequality to estimate the integral and we get $r_l$ of order $\Le (D^c)^{\frac1n}.$ This is not enough to recover {\bf P5}, since
the Banach norm $||(\hat{P}-\tilde{P}_l)\hat{h}||_{\mathcal{B}}$ is simply bounded by a constant as a consequence of the Lasota-Yorke inequality. Instead for $n=2$ the characterization of the total variation as the maximum of sectional variations along the coordinate axis is sufficient to get ({\bf P5}), and it was just used in \cite{CG}. By referring to (\ref{RC}) below, we can in fact bound the integral $\int |h{\bf 1}_{D_l^c}|d\Le$ by $1/2$ times the total variation of the density times the Lebesgue measure of the section of $D_l^c$ along one of the two coordinate axis (we are using here the corollary 2.1 in \cite{KELI}). But that sectional measure is of the same order of the Lebesgue measure of the whole $D_l^c$, just because we are on the unit square.   We therefore turn our attention to another functional space, the quasi-H\"older space, whose importance for expanding dynamical systems was stressed in the seminal works by Keller \cite{kel}  and Saussol \cite{S}.

We start by defining for all functions $h\in L^1(I^n)$ a semi-norm, which given two real numbers $\eps_0>0$ and $0<\alpha\le 1$, writes
$$
|h|_{\alpha}:=\sup_{0<\eps\le \eps_0}\frac{1}{\eps^{_\alpha}}\int \text{osc}(h, B_{\eps}(\ox)) d\Le,
$$
where $\text{osc}(h,A):=\text{Esup}_{\ox\in A}h(\ox)-\text{Einf}_{\ox\in A}h(\ox)$ for any
measurable set $A$.  We say that $h\in V_{\a}(I^n)$ if $|h|_{\a}<\infty$. Although the value of $|h|_{\a}$ depends on $\eps_0,$    the space $ V_{\a}(I^n)$ does not. Moreover the value of $\eps_0$ can be  chosen in order to satisfy a few geometric constraints, like distortion, and to guarantee the forthcoming bound (\ref{ETA2})\footnote{For explicit computations of $\eps_0$ on concrete examples, see \cite{S} and \cite{HV}; for the example (\ref{EX0}) below, that value was computed in Proposition 6 in \cite{TSU}.}. We equip $V_{\a}$ with the Banach norm
$$
||h||_{\a}:=|h|_{\a}+||h||_1,
$$
and from now on $V_{\a}$ will denote the Banach space $\mathcal{B}=(V_{\a}(I^n), ||\cdot||_{\a}).$ With the assumptions we put on the map $\hat{T}$, in particular for the nature and smoothness of the boundaries of the domains $U_k,$ it can be shown that the transfer operator $\P$ leaves $V_{\a}$ invariant with $\a=1,$ and moreover a Lasota-Yorke inequality ({\bf P1}) holds,  whenever

  \begin{equation}\label{ETA2}
\eta:=s_n+\frac{4s_n}{1-s_n}Z\frac{Y_{n-1}}{Y_n}<1,
\end{equation}
where $Y_n$ is the volume of the unit ball in $\mathbb{R}^n$ and $Z$ is the maximal number of the boundaries of the domains of local injectivity that  meet in one point, see \cite{S}.  Also, one can show that $\mathcal{B}$ can be continuously  injected into $L^{\infty}$ and in particular, \cite{S}, $||h||_{\infty}\le C_H ||h||_{\a}$, where $C_H=\frac{\max(1,\eps_0^{\a})}{Y_n \eps_0^n}$.

Our next step is to show that $\mathcal{B}$ is invariant under the perturbed operator $\tilde{P_l}$. By comparing with the computations in \cite{S}, we see that the new term we should take care of is:
$$
|h{\bf 1}_{D_l}|_{\a}=\sup_{0<\eps\le \eps_0}\frac{1}{\eps^{\a}}\int \text{osc}(h{\bf 1}_{D_l}, B_{\eps}(\ox)) d\Le.
$$
Using the results in \cite{S} and with $B_{\eps}(D_l)$ denoting the $\epsilon$-neighborhood of the set $D_l$\footnote{To be more precise we have $B_{\eps}(D_l):=\{\ox\in \mathbb{R}^n: \text{dist}(\ox, D_l)\le \eps\}.$} we have:
            $$
            \text{osc}(h{\bf 1}_{D_l}, B_{\eps}(\ox))\le \text{osc}(h, D_l\cap B_{\eps}(\ox)){\bf 1}_{D_l}+
            $$
            $$
            2 \left[ \text{Esup}_{B_{\eps}(\ox)\cap D_l}|h|\right]{\bf 1}_{ B_{\eps}(D)\cap (B_{\eps}(D_l^c))}(\ox).
            $$
            By integrating and dividing by $\eps^{-\a}$ we get
            $$
            |h{\bf 1}_{D_l}|_{\a}\le |h|_{\a}+\sup_{0<\eps\le \eps_0}\frac{2}{\eps^{\a}}\int_{B_{\eps}(\ox)\cap D}\sup|h(\ox)|{\bf 1}_{ B_{\eps}(D_l)\cap (B_{\eps}(D_l^c))}(\ox)d\Le \le$$
            $$
            |h|_{\a}+2||h||_{\infty}\frac{\Le( B_{\eps}(D_l)\cap (B_{\eps}(D_l^c))}{\eps^{\a}}.
            $$
            Before continuing we must say what really the set $D_l$ is in our case. Its complement, $D_l^c$ is given in (\ref{set}) and with the actual notation reads
            $$
            D_l^c=\{\ox\in I^n: \max_{i\neq j}|x_i-x_j|\le \nu_l\},
            $$
            where $\nu_l$ goes to zero when $l\rightarrow \infty.$ In this case it is easy to see that
            \begin{equation}\label{EDC}
            \Le( B_{\eps}(D_l)\cap (B_{\eps}(D_l^c))\le C_n\eps \nu_l,
             \end{equation}
             see Appendix 1 for the proof. Therefore we can continue the previous bound as:
            $$
             |h{\bf 1}_{D_l}|_{\a}\le  |h|_{\a}[1+2C_H C_n\eps^{1-\alpha}\nu_l ].
            $$
             This computation shows that $\mathcal{B}$ is preserved by $\tilde{P_l}$, but if we want to get a Lasota-Yorke inequality for it, and therefore satisfy ({\bf P3}), we should multiply $\eta$ by $(1+2C_HC_n\nu_l )$ and ask that $\eta(1+2C_HC_n\nu_l )<1, $ which is surely satisfied by taking $l$ large enough.   Alternatively, one could take higher iterates  of $\hat{T}$. In this case the backward images of $D_l$ will grow as well, but linearly with the power of the map and their contribution will be dominated by the exponential decay of the contraction factor.\\ As we said above property ({\bf P2}) requires that the invariant measure of the unperturbed map be mixing; we will give an explicit example below.

Since  quasi-H\"older functions $h$ are essentially bounded, it is easy to get Property ({\bf P4}) estimating as:
\begin{equation}\label{RC}
|\int (\P h-\tilde{P}_l h)d\Le|\le \int |h{\bf 1}_{D_l^c}|d\Le\le ||h||_{\infty}\Le(D_l^c)\le C_H||h||_{\alpha}\Le(D_l^c).
\end{equation}

  To check ({\bf P5}), we begin to observe that the Banach norm $||(\hat{P}-\tilde{P}_l)\hat{h}||_{\mathcal{B}}$ is bounded by a constant, say $\hat{C}$ depending on $\hat{h}$ as a consequence of the Lasota-Yorke inequality. Since the density is bounded away from zero, we immediately have $r_l ||(\P-\tilde{P}_l)\hat{h}||_{\mathcal{B}}\le \frac{C_H \hat{C}}{\hat{h}^{\mathrm{(inf)}}}\mu(D_l^c).$

\begin{example}\label{EX0}
\em We now give an easy example which satisfies {\bf P1} to {\bf P3} with $\mathcal{B}$ the space of quasi-H\"older functions; {\bf P4} and {\bf P5} follow from the above arguments and finally Property ({\bf P6}) will be proved in Section 5 under the additional assumption {\bf P0} and for a much larger class of maps. We stress that our example will be used for the numerical simulations in Section 7. Moreover the techniques we are using could be easily extendable to other transformations not necessarily affine.   As the one-dimensional map $T$ we will take $T(x)=3x$ mod$1.$ By  coupling $n$ of them  as in  (\ref{TM}) we get a piece-wise linear uniformly expanding higher dimensional map. We first notice that this map is not necessarily continuous on the $n$-torus, but it satisfies
the assumption ({\bf P0}) in Section 5. The Lasota-Yorke inequality (\ref{LLYY}) can be proved for $l$ large enough, say for $l>l_0$ if we verify the condition (\ref{ETA2}). If it does not hold for the map $\hat{T}$ it will be enough to get it for an iterate of $\T$ and this is surely possible thanks to Theorem 11 in Tsujii's paper \cite{TSU}, which holds for expanding piecewise linear maps whose locally domains of injectivity are bounded by polyhedra. The constants $\eta$ and $C$ in (\ref{LLYY}) depend in our case (local affine maps), simply on the contraction rate $s_n=3^{-n}$ to the power $l.$ The next step is to prove the bound (\ref{TN}).

This can be easily achieved by adapting our proofs of Proposition 4.3 in \cite{AFV}, or of Lemma 7.5 in \cite{HNTV}. The basic ingredients of such proofs are: (i) the control of the distance between the preimages of the same point $z\in I^n$ with the maps $\hat{T}_0$ and $\hat{T}$ (for a given $\gamma$); (ii) the distortion, involving the two determinants $|\det(D\hat{T_0})|$ and $|\det(D\hat{T})|$ (for a given $\gamma).$ By the structure of the map (\ref{TM}), one immediately sees that the distance at point (i) is of order $\gamma$ times a constant depending on the dimensionality $n$ of ambient space. The ratio of the determinants at point (ii) is instead of order $(1-\gamma)^n$ as it follows from the proof of Proposition \ref{pis} below. This is enough to obtain the bound (\ref{TN}); we left the details to the reader. We should finally check that the invariant density is bounded away from zero for the map $\hat{T}.$ We dispose of, at least, two criteria of {\em covering} type for that. The first is taken from Section 7.3.1 and Lemma 7.5 in our paper \cite{HNTV} and requires the existence of a  domain of local injectivity $U_k$ (see Section 2),  whose image is the full hypercube $I^n$. The second is described in Sublemma 5.3 in \cite{HV} and requires the so-called {\em topological exactness}, namely the existence for any
$\overline{x}\in I^n$ and $\epsilon>0,$ of an integer $N_{\epsilon} = N_{\epsilon}(x, \epsilon) > 0$ such that $\hat{T}^{N_{\epsilon}}B_{\epsilon}(\overline{x})=I^n.$ Both results rely on an interesting property of the quasi-H\"older functions, namely the existence of a ball where the (essential) infimum of such a function is bounded away from zero, see \cite{S}. We believe such covering  conditions are satisfied in our cases. As an example, we report the computation of the density for two and three  coupled maps;  it is also interesting to observe that the density does not oscillate too much in the vicinity of the diagonal, which is required by our assumption {\bf P8}, see Figure \ref{dens1} for $n=2$ and Figure \ref{dens} for $n=3$.

\end{example}

\section{Extreme values and localizations}\label{EVTL}
 In this section and in the next one, we apply EVT to the study of a few recurrence behaviors for our system of CML.


 There are, at least, two approaches to EVT. The first, which we
 call the pure probabilistic one (PPA) uses strong mixing properties to  get  fast decay of correlations for a suitable class of observables and to control short returns around a given point. It is worth mentioning  that the PPA covers cases where there is no spectral gap and therefore the   correlations do not decay exponentially fast, see for instance \cite{ERDS} for a rich variety of examples.

The second approach, developed by Keller \cite{KE} and which we name  the spectral approach (SA), is based on the  perturbation technique discussed in the preceding section and which allow us to get Gumbel's law directly by a smooth perturbation of the spectral radius of the operator $\tilde{P_l}$. We will show explicitly in Section 4 how this method works. The SA seems particularly adapted to investigate synchronization, while the PPA is not suited, for the moment, to study observables which become infinite on sets with uncountably many points, which is what happens when we consider synchronization (along the diagonal). As we have already pointed out in the previous section, the issue in the SA is to verify the properties {\bf P1}-{\bf P6}.

 Let us suppose the vector $\overline{z}$ is given. When the orbit of a point  $\overline{x}$ enters in a sufficiently small ball centered at $\overline{z}$ we will say that there is {\em localization} of the orbit around the point $\overline{z}$.

 Let us introduce the observable
 \begin{equation}\label{O1}
 \varphi(\overline{x}):=-\log\{\sum_{i=1}^n|x_i-z_i|\},
 \end{equation}
 and consider the maximum
 \begin{equation}\label{max}
 M_m(\overline{x}):=\max\{\varphi(\overline{x}), \varphi(\T\overline{x}),\dots \varphi(\T^{m-1}\overline{x})\}.
 \end{equation}
By adopting the point of view of EVT, we will fix a positive number $\tau$ and we will ask for the existence of  a sequence $u_m$ for which the following limit exists
\begin{equation}\label{TAU}
m\ \hat{\mu}(\varphi>u_m)\rightarrow \tau, \ m\rightarrow \infty.
\end{equation}
We will say that {\em the sequence $M_m$ has an Extreme Value Law, (EVL)}, if there exists a non-degenerate distribution function $H:\mathbb{R}\rightarrow [0,1],$ with $H(0)=0$ such that
\begin{equation}\label{EVT}
\hat{\mu}(M_m\le u_m)\rightarrow 1-H(\tau), \ m\rightarrow \infty.
\end{equation}
By using the expression of $\varphi$ we can rewrite (\ref{TAU}) as
\begin{equation}\label{TAU2}
m\ \hat{\mu}(U^{(n)}_m)\rightarrow \tau.
\end{equation}
where
\begin{equation}\label{DL1}
U^{(n)}_m:=\{\overline{x}\in I^n: \sum_{i=1}^n|x_i-z_i|\le \nu_m\},\  \text{with} \ \nu_m:=e^{-u_m}
\end{equation}
and consequently (\ref{EVT}) can be restated as
\begin{equation}\label{EVT21}
\hat{\mu}(\overline{x}\in I^n: \T^k(\overline{x})\notin U^{(n)}_m, k=0,\dots,m-1)\rightarrow 1-H(\tau).
\end{equation}
We call $\nu_m$ the {\em accuracy of the localization} and we use the symbol $a_c$ to denote it. Of course it depends on $m$, but as we will see soon, it is sometimes convenient to fix the value of $a_c$ and choose $m$ accordingly. If we see $\{\hat{T}^k\}_{k\ge 1}: I^n\rightarrow I^n$ as a   vector valued random variable on the space $\{I^n, \hat{\mu}\}$ associating to the point   $\overline{x}\in I^n$ its orbit,  then  the limit (\ref{EVT21}) could also be interpreted as the  {\em probability that each component $\{\hat{T}^k_i\}_{k\ge 1}$  is  localized with accuracy $a_c=e^{-u_m}$ around $z_i$  for the first time when $k>m.$}  In order to get the probability of such an event, we have  to insure a few assumptions, which were already anticipated in the previous section, and which will allow us  to apply Proposition 3.3 in \cite{AFV} that we restate in the following proposition:
\begin{prop}\label{P1}
Suppose that the system $(I^n, \hat{T}, \hat{\mu})$ has a unique absolutely continuous invariant and  mixing measure   $\hat{\mu}$ with density
bounded away from zero and exponential decay of correlations on an adapted Banach space.  Let $(X_0, X_1, \cdots)$ be the process given by $X_m = \phi\circ \hat{T}^m, m\in \mathbb{N},$
  where $\phi$  achieves a
global maximum at some points $\overline{z}$. Then we have an EVL for the maximum $M_m$ and:\\
(1) if $\overline{z}$ is not a periodic point, then the EVL is such that $H(\tau)=1-e^{-\tau}$;\\
(2) if $\overline{z}$ is a (repelling) periodic point of prime period $p$, then the EVL is
such that $H(\tau)=1-e^{-\theta\tau}$ , where the {\em extremal index (EI)}  is given by $\theta(\overline{z}) = 1-| \det D(\hat{T}^{p})(\overline{z})|^{-1}$.
\end{prop}
We notice that eventually (repelling) periodic points fall in part (1). Our observable (\ref{O1}) satisfies the assumption of the Proposition. On the other hand, by using Theorem 1.7.13 in \cite{MRGH}, we
have a sufficient condition to guarantee the existence of the limit (\ref{TAU})  for $0<\tau<\infty.$ Such a condition requires that $\frac{1-F(x)}{1-F(x-)}\rightarrow 1$, as $x\rightarrow u_F,$ where $F$ is
 the distribution function of $X_0$, the term $F(x-)$ in the denominator denotes
the left limit of $F$ at $x$ and $u_F=\sup\{x: F(x) < 1\}$. For the
observable just introduced $u_F = \infty$ and if the probability  $\hat{\mu}$ is not atomic at $\overline{z}$, then it is easy to conclude that $F$ is continuous at $\overline{z}$ and therefore the above ratio goes to $1$.\\

This general result will not allow us to explicitly compute the sequence $u_m.$ Let us take the  affine sequence: $u_m=\frac{y}{a_m}+b_m$, with $a_m>0,$ and $y\in \mathbb{R}$. This suggests that we redefine $u_m(y)$ as a one parameter family in $y.$  When the sequence  $\hat{\mu}(M_m\le u_m)=\hat{\mu}(a_m(M_n-b_m)\le y)$ converges to  a non-degenerate distribution function $G(y)$, in the point of continuity of the latter, then we have an EVL.
 It is a beautiful result of EVT, just related to the affine choice for the sequence $u_n$ \footnote{For other choices of the sequence $u_n$, see \cite{MRGH}.},  that such a $G(y)$ could be only of three types, called Gumbel, Fr\'echet and Weibull, see \cite{MRGH}, and what determines it in a particular situation is just the common distribution given by the  function $F.$\\

 For instance and in our case, if we suppose that the invariant measure behaves like Lebesgue,  $\hat{\mu}(U^{(n)}_m)=O(\nu_m^{n})$\footnote{ Actually we have $\hat{\mu}(U^{(n)}_m)=O(2^n\nu_m^{n})$, but the factor $2^n$ will become negligeable by taking large $m.$}, then $e^{-u_m}\sim \left(\frac{\tau}{m}\right)^{\frac{1}{n}},$ or equivalently $u_m\sim \frac1n\log m-\frac1n \log \tau$ and therefore {\em the probability of the first localization after $m$ iterations with $m$ large and with accuracy $a_c$  of order $\left(\frac{\tau}{m}\right)^{\frac{1}{n}}$, is $e^{-\tau},$} or equivalently $e^{-e^{-y}},$ having set $\tau=e^{-y}.$ The distribution function $e^{-e^{-y}}, y\in \mathbb{R}$ is just the Gumbel law. In this easy example $a_m=n, b_m=\frac1n \log m,$ but we used very crude approximation in estimating the $\hat{\mu}$-measure of the parallelepiped $U^{(n)}_m$ since we simply forgot the local density of the measure at the point $\overline{z}.$ Very often it is a difficult task to get an explicit expression for the scaling coefficients $a_m, b_m$. In a few cases one succeeds, see the results in \cite{ERDS}, Propositions 7.2.4, 7.4.1, 7.5.1.  Otherwise and for practical purposes,  the distribution function $\hat{\mu}(M_m\le y)$ is modeled, for $m$ sufficiently large, by the so-called {\em generalized extreme value (GEV)} distribution which is a  function depending upon three parameters $\xi\in \mathbb{R}, \mu\in \mathbb{R}, \sigma>0:$
$
\text{GEV}(y;\mu,\sigma, \xi)=\exp\left\{-\left[1+\xi\left(\frac{y-\mu}{\sigma}\right)\right]^{-1/\xi}\right\}.$

The parameter $\xi$ is called the tail index. When it is $0,$ the GEV corresponds to the Gumbel type, when the index is positive, it corresponds to a Fr\'echet and finally when it is negative, it corresponds to a Weibull. The parameter $\mu$ is called the location parameter and $\sigma$ is the scale parameter: for $m$ large the scaling constant $a_m$ is close to $\sigma^{-1}$ and $b_m$ is close to $\mu.$ \\

 The proof of Proposition (\ref{P1}) can be done with  the SA or the  PPA approaches and  the latter uses the approximation of our process with an i.i.d.  process, this being guaranteed   by the exponential rate of mixing of the measure $\hat{\mu}$ on functions in $\mathcal{B}.$ It is interesting to point out the dichotomy in the choice of the target point $\overline{z}:$ there is only two functional expressions for the distribution $H(\tau)$ and what determines such a difference is the possible periodicity of $\overline{z}$.  We now focus on the EI $\theta$. Suppose we have successive  entrances in the neighborhood of $\overline{z}$, namely consecutive occurrences of an exceedance of our threshold $u_n.$ We interpret it as a memory of the underlying random process, and we quantify it with the parameter $\theta$. In particular, see \cite{ERDS}, p. 34, when $\theta > 0$ and for most of the times, the inverse of the EI  defines the mean
number of exceedances  in a cluster of large observations, i.e., is the
mean size of the clusters. We now show that in our model and whenever the number of components of the lattice goes to infinity, the EI of periodic points goes to $1$, so  there are no  clusters   in the limit of an infinitely large lattice.

We now have (from now on we write the EI as $\theta_n$ to signify the dependence on $n$)
\begin{prop}\label{pis}
Let $\hat{T}$ be the CML  with $n$ sites given by (\ref{TM}) and   take $\gamma<1-\lambda.$ Fix $p\ge 1$, if $\oz^{(p)}_n\in I^n$ is a periodic point of prime period $p,$ the EI $\theta_n(\oz^{(p)}_n)$ satisfies
$$
\lim\limits_{n\to\infty}\theta_{n}(\oz^{(p)}_n)=1.
$$
\end{prop}
	\begin{proof} Recall the definition of the uncoupled dynamics
		\begin{eqnarray*}
		\T_0(\ox)=\big(T(x_{1}), T(x_{2}), \ldots, T(x_{n})\big)\qquad\forall\, \ox=(x_1,x_2,\dots,x_n)\in I^n,
	\end{eqnarray*}
		 and let $C_{\gamma}$ be the real $n\times n$ matrix, whose coefficients ${(C_\gamma)}_{ij}$ are defined by $(C_\gamma)_{ij}=\gamma/n$ if $i\neq j$ and
	$(C_\gamma)_{ij}=(1-\gamma)+\gamma/n$ if $i=j$. It is easy to check that $\hat{T}=\Phi_\gamma\circ \T_0$,
	where the coupling operator $\Phi_\gamma: I^n\to I^n$  is the linear map associated to the matrix $C_\gamma$ (i.e	$\Phi_\gamma(\ox):=C_\gamma \ox)$.

		Let $p\geq 1$, $\oz\in I^n$ and let us  compute the determinant of the Jacobian matrix of $\hat{T}^p$ evaluated in the point $\oz$. We have,
		$$
			\det(D_{\oz}\hat{T}^p)=\prod_{t=0}^{p-1}\det(D_{\hat{T}^{t}(\oz)}\hat{T})
			=\prod_{t=0}^{p-1}\det(D_{\T_0(\hat{T}^{t}(\oz))}\Phi_\gamma D_{\hat{T}^{t}(\oz)}\T_0)$$
$$
			=\prod_{t=0}^{p-1}\det(C_\gamma D_{\hat{T}^{t}(\oz)}\T_0)\\
			=\det(C_\gamma)^{p}\prod_{t=0}^{p-1}\det(D_{\hat{T}^{t}(\oz)}\T_0).
		$$
		It is an easy exercise in linear algebra to show that the determinant of the symmetric matrix $C_\gamma$ is $\det(C_{\gamma})=(1-\gamma)^{n-1}.$

		On the other hand $D_{\oz} \T_0$ is a diagonal matrix with diagonal entries $T'(z_1),\cdots,  T'(z_{n})$ and corresponding  Jacobian determinant   in $\oz$  given by $\det\big(D_{\oz}\T_0\big)=\prod^{n}_{k=1}T'(z_k)$. It follows that
		\[
		|\det\left(D_{\oz}\hat{T}^{p}\right)|=(1-\gamma)^{p(n-1)}\prod^{p-1}_{t=0}\prod^{n}_{k=1}|T'((\hat{T}^t(\oz))_k)|=(1-\gamma)^{p(n-1)}\prod^{p-1}_{t=0}\prod^{n}_{k=1}|T'(z^t_k)|.
		\]
		
		 According to \cite{AFV}, if $\oz^{(p)}_n$ is a periodic point of period $p$, the EI satisfies	
		$$\theta_{n}(\oz^{(p)}_n)=1-\frac{1}{|\det\big(D_{\oz^{(p)}_n} \hat{T}^{p}\big)|}.$$
		
		Since
		\[
		|\det\left(D_{\oz^{(p)}_n}\hat{T}^{p}\right)|\geq(1-\gamma)^{p(n-1)}(\frac{1}{\lambda})^{np}=\left(\frac{(1-\gamma)^n(\frac{1}{\lambda})^n}{(1-\gamma)}\right)^p,
		\]
		and as $\gamma<1-\lambda$,  we have that $\lim\limits_{n\to\infty}\theta_{n}(\oz^{(p)}_n)=1$.
	\end{proof}

\subsection{Random perturbations}
There is  another situation which produces an extremal index equal to $1$. We can perturb the map $\hat{T}$ with additive noise, see \cite{AFV},  by defining a family of maps $\T_{\o}=\T+\o$, with each vector  $\o$ belonging to the set $\Omega$ and chosen in such a way that each $\T_{\o}$ sends $I^n$ into itself. The iteration of $\hat{T}$ will be now replaced by the concatenation
\[
\hat{T}^n_{\oom}:=\hat{T}_{\o_n}\circ \cdots \circ \hat{T}_{\o_1},\text{ with } \ \oom:=(\o_1, \cdots, \o_n, \cdots)\in \Omega^{\mathbb{N}},
\]
and the $\o_k$ chosen in an i.i.d. way in $\Omega$ according to some (common) distribution $\mathbb{P}.$ If we now take any measurable real observable $\phi$, the process $\{\phi\circ \hat{T}^n_{\oom}\}_{n\ge 1}$ will be stationary with respect to the product measure $\hat{\mu}_s\times \mathbb{P}^{\mathbb{N}},$ where $\hat{\mu}_s$ is the so-called {\em stationary measure}, verifying,  for any real measurable bounded function $f$: $\int f d\hat{\mu}_s=\int f\circ \hat{T}_{\omega}d\hat{\mu}_s:$  see \cite{ERDS} Chapter 7 for a general introduction to the matter. We call the couple $\{\hat{T}_{\oom}, \hat{\mu}_s\times \mathbb{P}^{\mathbb{N}}\}$ a {\em random dynamical system}. In the framework of EVT we could therefore consider the process $\{X_{m,\oom}(\cdot)\}_{n\ge1}=\{\phi \circ \hat{T}^m_{\oom}(\cdot)\}_{n\ge1}$, where $\phi$ is the observable introduced in (\ref{O1}), and consider  accordingly the distribution of the maximum (\ref{max}) with respect to the probability measure $\hat{\mu}_s\times \mathbb{P}^{\mathbb{N}}.$ By adopting for $\hat{T}$ the same assumptions as in Proposition \ref{P1}, it is not difficult  to show that $\hat{\mu}$ is equivalent to Lebesgue and we finally  proved in \cite{AFV}, Corollary 4.4, that for {\em any} choice of the target point $\oz$, an extreme value distribution holds with $H(\tau)=1-e^{-\tau}.$

\section{Extreme values and synchronization}\label{EVTS}
 We now introduce a new observable which  allows us to consider synchronization of the $n$ components of an initial state  iterated by $\hat{T}.$ Let us therefore define
\begin{equation}\label{S1}
 \psi(\overline{x}):=-\log\{\max|x_i-x_j|, i\neq j: i,j=1,\dots,n\}
\end{equation}
 and consider the maximum
 $$
 M_m(\overline{x}):=\max\{\psi(\overline{x}), \psi(\T\overline{x}),\dots \psi(\T^{m-1}\overline{x})\}.
 $$
By adopting the point of view of EVT, we  fix again a  positive number $\tau$ and we  ask for a sequence $u_m$ for which the following limit exists
$
m\ \hat{\mu}(\psi>u_m)\rightarrow \tau, \ m\rightarrow \infty.
$
We say again that {\em the sequence $M_n$ has an Extreme Value Law}, if there exists a non-degenerate distribution function $H:\mathbb{R}\rightarrow [0,1],$ with $H(0)=0$ such that $
\hat{\mu}(M_m\le u_m)\rightarrow 1-H(\tau), \ m\rightarrow \infty.$
By using the expression of $\psi$ we can rewrite (\ref{TAU}) as
\begin{equation}\label{TAU2}
m\ \hat{\mu}(S^{(n)}_m)\rightarrow \tau
\end{equation}
\begin{equation}\label{set}
S^{(n)}_m:=\{\overline{x}\in I^n: \max_{i\neq j}|x_i-x_j|\le \nu_m\},\  \text{where} \ \nu_m:=e^{-u_m}
\end{equation}
and consequently (\ref{EVT}) can be restated as
\begin{equation}\label{EVT2}
\hat{\mu}(\overline{x}\in I^n: \T^k(\overline{x})\notin S^{(n)}_m, k=0,\dots,m-1)\rightarrow 1-H(\tau).
\end{equation}
The limit (\ref{EVT2}) could also be interpreted as the  {\em probability that the $n$ components have synchronized for the first time after $m$ iterations with accuracy $a_c$  of order $e^{-u_m}$}.

We cannot  use the PPA to prove the existence of the limit (\ref{EVT2}). The reason is that our new observable becomes infinite on a line (the diagonal), and for the moment rigorous results are avalaible when the set of points where the observable is maximised is at most countable, see \cite{FNE} for a discussion of these problems.

The SA will bypass that issue by using the  Banach space $\B$ given by quasi-H\"oder functions, since for such a space we can check properties {\bf P1}-{\bf P5}.  Nevertheless there is still a problem remaining, namely prove the existence of the limits (\ref{Q3}). We will return to that in the next section.

We now show how to get the asymptotic distribution functions of the extreme value theory by using the SA. Let us begin by rewriting
 the maximum given in (\ref{EVT2}) using the density $\hat{h}$ of the measure $\hat{\mu}$:
\begin{equation}\label{KT}
\hat{\mu}(M_n\le u_m)=\int \hat{h}(\overline{x}){\bf 1}_{(S^{(n)}_m)^c}(\overline{x}){\bf 1}_{(S^{(n)}_m)^c}(\T(\overline{x}))\dots {\bf 1}_{(S^{(n)}_m)^c}(\T^{m-1}(\overline{x}))d\Le=\int \tilde{P}^m_m(\hat{h}) d\Le,
\end{equation}
where, from now on,
$$
\tilde{P}_m(\cdot):=\P({\bf 1}_{(S^{(n)}_m)^c}\cdot).
$$
Notice that $(S^{(n)}_m)^c$ plays the role of the set $D_l$ in Section 2. By invoking the spectral representation (\ref{deco}) we have with obvious interpretation of the symbols
$$
\int \tilde{P}^m_m(\hat{h}) d\Le=\rho_m^m\int \hat{h}d\hat{\mu}_m+\int Q_m^m \hat{h}d\Le,
$$
where $\int \hat{h}d\hat{\mu}_m\rightarrow \int \hat{h}d\Le=1,$ as $m\rightarrow \infty,$ and the spectral radius of $Q_m$ is strictly less than $\rho_m.$
We now need  to bound $\rho_m,$ the largest eigenvalue of $\tilde{P}_m$, for increasing $m$ and it is given by  (\ref{Q1}).
Let us now denote the exponent $\theta_{\Delta}$  the EI along the diagonal set $\Delta:=\{\overline{x}\in \mathbb{R}^n; x_1=x_2,\cdots=x_n\}$ and its existence will follow if we prove  limit  (\ref{Q3}). We then write:
$$
1-\rho_m=\hat{\mu}(S^{(n)}_m) \theta_{\Delta} (1+o(1)), \ \text{in the limit} \ m\rightarrow \infty,
$$
then
\begin{equation}\label{PIV}
\int \tilde{P}^m_m(\hat{h}) d\Le=e^{-(\theta_{\Delta} m \hat{\mu}(S^{(n)}_m)+m o(\hat{\mu}(S^{(n)}_m))}\int \hat{h}d\hat{\mu}_m+O(\rho_m^{-m}||Q_m^m||_{\mathcal{B}})
\end{equation}
which converges to $e^{-\tau \theta_{\Delta}}$ under the assumptions on $\hat{\mu}_m$, the spectral radius of $Q_m$ and the condition (\ref{TAU2}). From now on  we will simply write $\theta_n$ for the EI along the diagonal set for lattices with $n$ components.\\

We now return to (\ref{EVT2}) since we now know that $1-H(\tau)=e^{-\theta_n \tau}.$ If we suppose that $\hat{\mu}(S^{(n)}_m)=O(\nu_m^{n-1})$\footnote{ Actually this is a very crude approximation. In fact what  is possible to prove easily is an upper bound on the Lebesgue measure of the domain $\{\overline{x}\in I^n, |x_i-x_j|<\nu_m: i\neq j\}$  which is simply $(2\nu_m)^{n-1}.$ We sketch the argument for $n=3$. In this case, the measure we are looking for is $\int dx_1 \int dx_2 {\bf 1}_{\{|x_1-x_2|\le \nu_m\}}(\ox) \int dx_3 {\bf 1}_{\{|x_1-x_3|\le \nu_m\}}(\ox){\bf 1}_{\{|x_2-x_3|\le \nu_m\}}(\ox).$ The last integral will contribute with $2\nu_m$ and so the second one.}, then $e^{-u_m}\sim \left(\frac{\tau}{m}\right)^{\frac{1}{n-1}}$ and therefore {\em the probability of the first synchronization after $m$ iterations  with accuracy $a_c\sim \left(\frac{\tau}{m}\right)^{\frac{1}{n-1}}$, is $e^{-\theta_n\tau}.$}\footnote{We defer to the discussion after Proposition 3.1 for the validity of this argument and its approximations.} If the components of the vector $\T^k(\overline{x})$ are seen as the positions of different particles on a  lattice, we have a quantitative estimate of the probability of synchronization of the lattice after a prescribed time and with a given accuracy.
\begin{example}\label{EX}

\begin{itemize}
\item
{\em Ex. 1.} Suppose  we use the data in Section 7, with an EI  $\theta_3\sim 1-(\frac{10}{27})^2\sim 0.86,$ having chosen $\lambda=1/3$ and $\gamma=0.1,$ and take $3$ particles  each living on the unit interval. If we want to synchronize them with a probability larger than $1/2$ and  an accuracy $a_c=0.01$ before $m$ iterations, then  we have to iterate the lattice around $m=8 \ 100$ times.
\item {\em Ex. 2.} Analogously, if we want to observe with a probability larger than $1/2$ the synchronization of $100$ particles  each living on the unit interval  with an accuracy $a_c = 0,01$ and  before $m$ iterations of the CML, then $m$ has to be larger than $100^{100}.$
\end{itemize}
\end{example}
\section{Computation of the extremal index}
The extremal index is given by formula (\ref{Q2}). Keller showed in \cite {KE} that it coincides with that given in Proposition \ref{P1} for the process $X_m=\phi\circ\hat{T}^m$ and the proof is exactly the computation we performed in the previous section.   As we said in the introduction, the rigorous computation of the EI for two coupled maps was given in \cite{CG}. Their map was slightly different from ours in the sense that for the $i$-th component the averaged  term $\frac{\gamma}{n}\sum_{j=1}^n T(x_j)$ does not contain the contribution of $T(x_i).$ They first observed that in (\ref{Q3}), all the $q_k$  but $q_0$, are zero due to the fact that the diagonal is invariant and $q_0$ reads:
\begin{equation}\label{q0}
q_0=\lim_{m\rightarrow \infty}\frac{\hat{\mu}(S^{(2)}_m\cap \T^{-1}S^{(2)}_m)}{\hat{\mu}(S^{(2)}_m)}
\end{equation}
This quantity  was explicitly computed  giving the formula \cite{CG}:
$$
            \theta_2=1-\frac{1}{1-2\delta}\frac{1}{\int \hat{h}(x,x)  dx}\int \frac{\hat{h}(x,x)}{|DT(x)|}dx,
            $$
where the density $\hat{h}$ has bounded variation and  for almost every $x\in I$ the
  value $\hat{h}(x, x)$ is the average of the limits of $\hat{h}(x-u, x+u)$ and $\hat{h}(x + u, x-u)$ as $u\rightarrow 0.$ \\

 We get a similar result  and still for $n=2$, with a modification due to the fact that
our map is different, see formula (\ref{EEIII}) in the remark below. Instead   the density along the diagonal is defined again as a bounded variation function.
It seems difficult to extend such a result in higher dimensions without
 much stronger assumptions. Before doing that, we will explore how the EI
 $\theta_n$ behaves for large $n$ in a quite general setting with the
 objective to show that for large $n$ such an index approaches $1$ and therefore
 the Gumbel's law will emerge as the extreme value distribution.
 \\We will  index with $n$ the invariant densities  $\hat{h}_n,$
   while we continue to use the symbol $\hat{\mu}$ for the invariant measure,
    despite the fact that $\hat{\mu}$ depend on $n$ too, {\em via} the density $\hat{h}_n.$
Our next objective is to show that all the $q_k$ but $q_0$ are zero. Such a result is
    claimed in \cite{CG} in dimension $2$ and  without proof; we sketch it  below for the reader's convenience in any dimension and asking for a few assumptions.

    We first notice that the quantities $q_{k,l}$ introduced in (\ref{Q3}), read:
\begin{equation}\label{PIVA}
q_{k,l}:=\frac{\int (\P-\tilde{P}_l)\tilde{P}_l^k(\P-\tilde{P}_l)(\hat{h})d\Le}{\hat{\mu}(D^c_l)}=\hat{\mu}_{D^c_l}
\{x\in D^c_l :{\bf t}_{D^c_i}(x)=k+1\}
\end{equation}
where $\hat{\mu}_{D^c_l}$  is the conditional measure to $D^c_l,$ and ${\bf t}_{D^c_i}(x)$ denotes the first return
 time of the point $x\in D^c_l$ to $D^c_l$ (we will come back on this equality in the next section). Additional properties are necessary; for that let us first  denote with $V_{\varepsilon}(\Delta)$ an $\varepsilon$-neighborhood
  of the diagonal $\Delta.$

\begin{itemize}
\item {\bf P01}: The boundaries of the domains of local injectivity $U_1,\cdots, U_q$ (see Section 2) are union of  finitely many discontinuity surfaces $\mathcal{D}_j, j=1,\cdots,p$\footnote{We observe that the map could be continuous on such boundaries, but the first derivative surely is not.}, which  are   co-dimension $1$ embedded submanifolds.  We denote by $\mathcal{D}$ the union of those discontinuity sets. Moreover $\forall \varepsilon>0$ and $k\in \mathbb{N},$  let us denote with $F_{d, \varepsilon, k}$ the set of points $x\in V_{\varepsilon}(\Delta)$ for which there is  a neighborhood $\mathcal{O}(x)$ such that $\mathcal{O}(x)\cup \Lambda\neq \emptyset,$  and  $\mathcal{O}(x)\cap (\mathcal{D}\cup \hat{T}^{-1}(\mathcal{D})\cup \cdots \hat{T}^{-k}(\mathcal{D}))=\emptyset.$  We require the existence of a constant $C_k$ independent on $\varepsilon$  such that $\hat{\mu}(F_{d, \varepsilon, k}^c)\le C_k \sigma(\varepsilon)\hat{\mu}(V_{\varepsilon}(\Delta))$, where $\sigma(\varepsilon)$ goes to zero when $\varepsilon\rightarrow 0.$\\

\item {\bf P02}: Let us denote with $G_{d, \varepsilon}$ the set of points in $V_{\varepsilon}(\Delta)$ for which   the   segment of minimal lenght connecting one of this point to the diagonal intersects
one component of $\hat{T}{\mathcal D}_j.$ For $\varepsilon$ small enough, we will assume that
  there is a constant $C_d$ independent of $\varepsilon$  such that $\hat{\mu}(G_{d, \varepsilon})\le C_d \ \kappa(\varepsilon) \ \hat{\mu}(V_{\varepsilon}(\Delta)),$ where $\kappa(\varepsilon)$ goes to zero when $\varepsilon\rightarrow 0.$
\end{itemize}
\begin{rmk}

 The condition {\bf P01} means that for a large portion of points in the vicinity  of the diagonal, we can find a neighborhood which intersects the diagonal but does not cross the discontinuity lines up to a certain order. The condition {\bf P02} means that the piece of $V_{\varepsilon}(\Delta)$ which is crossed by an element of $\hat{T}{\mathcal D}_j$ has a  length along the direction of $\Delta$ of order $\kappa(\varepsilon).$
Both situations happen when the crossing of the discontinuities are ``transversal": it is easy to produces pictures in dimension $n=2$ and $n=3.$ See Figure \ref{DIPO} for {\bf P01} and Figure \ref{DIS} for {\bf P02}. In both cases we took $\varepsilon$ small enough in such a way that the discontinuity behaves locally, when it intersects $V_{\varepsilon}(\Delta)$, as a line for $n=2$ and as a plane for $n=3$; moreover $\kappa(\varepsilon)=O(\varepsilon).$
We notice that condition  {\bf P02} requires the control only  of the first images of $\mathcal{D}$ and also it is not necessary if the map $\hat{T}$ is onto on each $U_l.$
\end{rmk}

We sketch the argument for $k=1$, the others being similar. By replacing $l$ with $m$ in (\ref{PIVA}) we show that:
\begin{lemma}
The quantity
\begin{equation}\label{FFF}
 \frac{\int (\P-\tilde{P}_m)\tilde{P}_m(\P-\tilde{P}_m)(\hat{h})fd\Le}{{\hat{\mu}(S_m^{(n)})}}=\frac{\hat{\mu}(S^{(n)}_m\cap\hat{T}^{-1}(S^{(n)}_m)^c
 \cap \hat{T}^{-2}S^{(n)}_m)}{\hat{\mu}(S_m^{(n)})}
\end{equation}
goes to zero when $m\rightarrow \infty.$
\end{lemma}
\begin{proof}
 Let us  take a point $x\in F_{d, \varepsilon, 2}$. With these assumptions $\hat{T}$ and $\hat{T}^2$
are open maps on $\mathcal{O}(x)$. In particular, $\hat{T}^2(\mathcal{O}(x))$
will be included in the interior of one of the $U_l$ and
 it will intersect $\Delta$ by the forward invariance of the latter. We now suppose that $\hat{T}^2(x)$ is in $V_{\varepsilon}(\Delta)$ and we try to prove that $\hat{T}(x)$ must be in $V_{\varepsilon}(\Delta)$ too. Let us call $D_*$ the domain of the function  $\hat{T}^{-1}_*$, namely the inverse branch of the map sending $\hat{T}(x)$ to  $\hat{T}^2(x). $ If the distance between $\hat{T}^2(x)$ and any point $z\in \hat{T}^2(\mathcal{O}(x))\cap \Delta,$  such that the segment $[ \hat{T}^2(x), z]$ is included in $D_*,$ is less than $\varepsilon$, we have done since $\text{dist}(\hat{T}^{-1}_*(z),
 \hat{T}^{-1}_*(\hat{T}^2(x))=\text{dist}(\tilde{z},\hat{T}(x))\le \lambda \varepsilon,$  where $\tilde{z}=\hat{T}^{-1}_*(z)\in \Delta.$ Notice that such a point $z\in \Delta$ should not be necessarily in $\hat{T}^2(\mathcal{O}(x)),$ provided the segment $[ \hat{T}^2(x), z]\in D_*$ and $\text{dist}(z,\hat{T^2}(x))\le \ \varepsilon.$ What could prevent the latter conditions to happen is the presence of the boundaries of the domains of definition of the preimages of $\hat{T}$, which are the images of $\mathcal{D}.$
We should therefore avoid that $\hat{T}^2(x)$ lands in the set $G_{d, \varepsilon},$  which means we have to discard those points  $x\in V_{\varepsilon}(\Delta)$ which are in  $\hat{T}^{-2}G_{d, \varepsilon},$ and, by invariance, the measure of those point is bounded from above by $C_d \ \kappa(\varepsilon) \hat{\mu}(V_{\varepsilon}(\Delta)).$ We now choose $\nu_m<\varepsilon$ and work directly with the sets $S_m^{(n)}.$ The points which are not in $\hat{T}^{-2}G_{d, \nu_m }\cap S_m^{(n)}\cap F_{d, \nu_m, 2}^c $ gives zero contribution  to the quantity $\hat{\mu}(S^{(n)}_m\cap\hat{T}^{-1}(S^{(n)}_m)^c
 \cap \hat{T}^{-2}S^{(n)}_m)$, while the measure of the remaining points divided by $\hat{\mu}(S_m^{(n)})$ goes to zero for $m$ tending to infinity.

 \end{proof}

\begin{prop}\label{PROPO1}
Let us suppose our CML satisfies properties {\bf P1}-{\bf P5} on a Banach space $\B$ with $\lambda=\inf|DT|^{-1}<1-\gamma,$ the density $\hat{h}_n\in L^{\infty}$ and $\hat{h}_n^{(\mathrm{inf})}:=\inf_{I^n}\hat{h}_n>0.$ Then
$$ \limsup_{m\rightarrow \infty}\frac{\hat{\mu}(S^{(n)}_m\cap \hat{T}^{-1} S^{(n)}_m)}{\hat{\mu}(S^{(n)}_m)}\le\frac{
\lambda^{n-1}||\hat{h}_n||_{\infty}}{(1-\gamma)^{n-1}\hat{h}_n^{(\mathrm{inf})}}
$$
\end{prop}
\begin{rmk}
The upper bound makes  sense of course  when the right hand side of the above inequality is less or equal to $1.$  Moreover the EI $\theta_n$ will converge to $1$, under the additional {\bf P0} assumptions,  when $n\rightarrow \infty$ if the ratio $\frac{ ||\hat{h}_n||_{\infty}}{\hat{h}_n^{(\mathrm{inf})}}$ does not grow faster than $\upsilon^{n-1}$ with $\upsilon>(\lambda/(1-\gamma))^{-1}.$
\end{rmk}
\begin{proof}
We start  by writing
  $$\hat{\mu}\big(\mathcal{S}^{(n)}_{m}\cap\hat{T}^{-1}\mathcal{S}^{(n)}_{m}\big)= \int_{I^n} d\overline{x} \hat{h}_n(\overline{x}) \textbf{1}_{\mathcal{S}^{(n)}_{m}}(x)\textbf{1}_{\mathcal{S}^{(n)}_{m}}(\hat{T}\overline{x})$$
  $$ =\int_I dx_{1}\int_{I^{n-1}} dx_{2}\dots dx_{n} \hat{h}_n(x_1,\dots,x_n){\textbf{1}}_{\mathcal{S}^{(n)}_{m}}(\overline{x})\cdot $$$${\textbf{1}}_{\mathcal{S}^{(n)}_{m}}\Big({\tiny(1-\gamma)T(x_1)+ \frac{\gamma}{n}\sum_{i=1}^{n}T(x_{i}),\dots, (1-\gamma)T(x_n) +\frac{\gamma}{n}\sum_{i=1}^{n}T(x_{i})}\Big).$$
  W now have  to reduce the domain of integration of $I\ni x_1$ in two steps: the first, changing $I$ into $I'_m,$ consists in removing intervals of length $2\nu_m$ on the left and on the right on each boundary point of the $A_l, l=1,\cdots, q.$ Clearly the difference between the integrals over $I$ and $I'_m$ will converge to zero when $m\rightarrow \infty$ since the integrand functions are bounded (remember the density is in $L^{\infty});$  this argument is  made more precise in Appendix 2 together with the reason of that reduction. For the moment we simply write ${\bf I}(I \setminus I'_m)$ for the integral over $I\setminus I'_m.$
	 	By introducing the operator $P_l$ acting on the variable $x_l$, $l\ge 2$, we could continue  as:	 	$$\hat{\mu}\big(\mathcal{S}^{(n)}_{m}\cap\hat{T}^{-1}\mathcal{S}^{(n)}_{m}\big)= {\bf I}(I\setminus I'_m)+ \int_{I'_m} dx_{1}\int_{I^{n-1}} dx_{2}\dots dx_{n} P_{2}\circ\dots\circ P_{n} \Big[\hat{h}_n(x_1,\dots,x_n){\textbf{1}}_{\mathcal{S}^{(n)}_{m}}(\overline{x})\Big]\cdot	
	 	 $$
	 	$$ {\textbf{1}}_{\mathcal{S}^{(n)}_{m}}\Big({\tiny(1-\gamma)T(x_1) +\frac{\gamma}{n}\big(T(x_{1})+x_{2}+\dots+x_{n}\big),\dots, (1-\gamma)x_n +\frac{\gamma}{n}\big(T(x_{1})+x_{2}+\dots+x_{n}\big)}\Big).$$
If we now introduce the sets
$$
S^{(n)}_{m,\gamma}(Tx_1)=\{(x_2,x_3,\cdots,x_n)\in I^n: \mid T(x_1)-x_{j}\mid  \leq  \frac{\nu_{m}}{1-\gamma}, j= 2,\dots,n,\
	 	    \mid x_{i}-x_{j}\mid  \leq  \frac{\nu_{m}}{1-\gamma}\quad i\not=j\not=1\} ,
$$
and
$$
S^{(n)}_{m}(x_1)=\{(x_2,x_3,\cdots,x_n)\in I^n: |x_1-x_j|\le \nu_m, j=2,\cdots,n, \mid x_{i}-x_{j}\mid  \leq  \nu_{m}\quad i\not=j\not=1\},
$$
we have
$$
\frac{\hat{\mu}(S^{(n)}_m\cap \hat{T}^{-1} S^{(n)}_m)}{\hat{\mu}(S^{(n)}_m)}\le
$$
$$
\frac{\int_{I'_m} dx_1\int_{S^{(n)}_{m,\gamma}(Tx_1)} dx_{2}\dots dx_{n} P_{2}\circ\dots\circ P_{n} \Big[\hat{h}_n(x_1,\dots,x_n){\textbf{1}}_{\mathcal{S}^{(n)}_{m}}(\overline{x})\Big]+{\bf I}(I\setminus I'_m)}{\int_{I'''_m} dx_1\int_{S^{(n)}_{m}(x_1)}dx_2\cdot dx_n\hat{h}(x_1,\cdots,x_n)}
$$
We reduced the domain of integration in the integral in the denominator from $I$ to $I'''_m:$ this kind of reduction will also affect $I'_m$ and it will be explained in the Appendix 2.
 Let us now consider for simplicity the structure of the operators when $n=3$:
	 	 $$P_{2}\circ P_{3}\big[\hat{h}_3(x_{1}, x_{2}, x_{3})\textbf{1}_{\mathcal{S}^{(3)}_{m}}(x_1,x_2,x_3)\big]=$$
\begin{equation}\label{SC}
\sum\limits_{j}\sum\limits_{k}\frac{\hat{h}_3(x_{1},T^{-1}_{j}x_{2}, T^{-1}_{k}x_{3})
         \textbf{1}_{\mathcal{S}^{(3)}_{m}}(x_{1},T^{-1}_{j}x_{2}, T^{-1}_{k}x_{3})}{\mid DT(T^{-1}_{j}x_2)\mid \mid DT(T^{-1}_{k}x_3)\mid}\textbf{1}_{TA_j}(x_2)\textbf{1}_{TA_k}(x_3),
         \end{equation}
	 	 where $\{A_k\}$ denotes the intervals of monotonicity of the map $T$.   The preceding  constraints and the assumption $\gamma<1-\lambda$ imply that:
	 	 $
	 	 \mid T_{j}^{-1}x_{2}-x_1 \mid <\nu_m,
	 	 	\mid T_{k}^{-1}x_{3} -x_1 \mid < \nu_m$.
Since
	 	  the original partition is finite, if we take first $m$ large enough and having removed the intervals of length $2\nu_m$ around the boundary point of the domain of monotonicity of $T$,  for any $x_1\in I'_m$ there will be only one preimage which can contribute in each sum.  By generalizing to $n$ components we could therefore bound the term (\ref{SC}) by $\lambda^{n-1}||h||_{\infty}$. Moreover a simple geometrical inspection shows that the Lebesgue measures of the sets $S^{(n)}_{m,\gamma}(Tx_1)$ and $S^{(n)}_{m}(x_1)$ are independent of the point $x_1$ and also the ratio of the two measures is independent of $m$ and  gives
 \begin{equation}\label{CAR}
 \frac{\Le (S^{(n)}_{m,\gamma})}{\Le(S^{(n)}_{m})}=\frac{1}{(1-\gamma)^{n-1}},
  \end{equation}
  see Appendix 2.
We therefore get
 \begin{equation}\label{ing of remp}
	 	\frac{\hat{\mu}\big(\mathcal{S}^{(n)}_{m}\cap\hat{T}^{-1}\mathcal{S}^{(n)}_{m}\big)}{\hat{\mu}(\mathcal{S}^{(n)}_{m})}\leq \frac{\Le (S^{(n)}_{m,\gamma})\lambda^{n-1}||\hat{h}_n||_{\infty}+{\bf I}(I\setminus I''_m)}{\Le(S^{(n)}_{m})\hat{h}_n^{(\text{inf})}}.
	 	 \end{equation}
We now notice that ${\bf I}(I\setminus I''_m)$ can be immediately bounded by $||\hat{h}||_{\infty}\Le (S^{(n)}_{m,\gamma}) \Le(I\setminus I''_m)$. This allows us to factorize the term $\Le (S^{(n)}_{m,\gamma})$ in the denominator and divide it by $\Le(S^{(n)}_{m}).$ By taking the $\limsup$ we finally get our result.
\end{proof}
We can now strengthen the previous result by adding  further assumptions.  We start first with a stronger  hypothesis on the invariant density which we will relax later on.
\begin{itemize}
\item {\bf P7} The density $\hat{h}$ is continuous on $I^n.$
\end{itemize}
This condition is for instance satisfied in the uncoupled case for smooth and locally onto maps $T$ of the unit circle.
\begin{prop}\label{propp}
Let us suppose that our CML satisfies properties {\bf P1}-{\bf P5} and {\bf P7} on a Banach space $\mathcal{B}$ with $\lambda=\inf|DT|^{-1}<1-\gamma,$  then
$$ \lim_{m\rightarrow \infty}\frac{\hat{\mu}(S^{(n)}_m\cap \hat{T}^{-1} S^{(n)}_m)}{\hat{\mu}(S^{(n)}_m)}=\frac{1}{(1-\gamma)^{n-1}}\frac{\int_I \frac{\hat{h}_n(x,\cdots, x)}{|DT(x)|^{n-1}}dx }{\int_I \hat{h}_n(x,\cdots,x)dx}.
$$
\end{prop}
\begin{proof}
We will write the proof for $n=3$, the generalization being immediate, and this will allows us to use the simple formulas in the previous demonstration. By the same arguments in the latter and by denoting with $T_{x_1}^{-1}$ the inverse branch of $T$ such that $T_{x_1}^{-1}(T(x_1))=x_1,$ we have
\begin{equation}\label{STOC}
\hat{\mu}\big(\mathcal{S}^{(3)}_{m}\cap\hat{T}^{-1}\mathcal{S}^{(3)}_{m}\big)=\int_{I''_m} dx_1\int_{S^{(3)}_{m,\gamma}(Tx_1)}\frac{\hat{h}_3(x_1, T_{x_1}^{-1}x_2, T_{x_1}^{-1}x_3)}{|DT(T_{x_1}^{-1}x_2)||DT(T_{x_1}^{-1}x_3)|}dx_2dx_3+{\bf I}(I\setminus I''_m)
\end{equation}
and we have a lower bound for $\hat{\mu}\big(\mathcal{S}^{(3)}_{m}\cap\hat{T}^{-1}\mathcal{S}^{(3)}_{m}\big)$ without the ${\bf I}(I\setminus I''_m)$ term. We call ${\bf I}(I''_m)$ the first integral on the right hand side.

Since $\hat{h}_3$ is continuous on $I^3$  and therefore uniformly continuous, having fixed $\tilde{\eps}>0$, it will be enough to choose $\nu_m$ small enough (remember  that
$|T_{x_1}^{-1}x_2-x_1|\le \nu_m$, $|T_{x_1}^{-1}x_3-x_1|\le \nu_m$),   to have
$
\hat{h}_3(x_1, T_{x_1}^{-1}x_2, T_{x_1}^{-1}x_3)=\hat{h}_3(x_1,x_1,x_1)+O(\tilde{\eps}).
$

For the derivative we can use the fact that our map is $C^2$ on the interior of the $A_l, l=1,\cdots,q$ and extendable with continuity on the boundaries to get by the mean value  theorem
$$
DT(T_{x_1}^{-1}x_2)=DT(x_1)+D^2T(\hat{x}_2)|T_{x_1}^{-1}x_2-x_1|, \
DT(T_{x_1}^{-1}x_3)=DT(x_1)+D^2T(\hat{x}_3)|T_{x_1}^{-1}x_3-x_1|
$$
where $\hat{x}_2$ belongs to the interval with endpoints $T_{x_1}^{-1}x_2$ and $x_1$, and $\hat{x}_3$ belongs to the interval with endpoints $T_{x_1}^{-1}x_3$ and $x_1$ and these two intervals are in the domains where $T$ is locally injective.   By inserting these formulas in the definition of ${\bf I}(I''_m)$ we have:
$$
{\bf I}(I''_m)=\int_{I''_m} dx_1 \frac{\hat{h}_3(x_1,x_1,x_1)}{|DT(x_1)|^2}\int_{S^{(3)}_{m,\gamma}(Tx_1)}\frac{dx_2dx_3}{ [1+\frac{D^2T(\hat{x}_2)}{DT(x_1)}|T_{x_1}^{-1}x_2-x_1|][1+\frac{D^2T(\hat{x}_3)}{DT(x_1)}|T_{x_1}^{-1}x_3-x_1|]}+$$
$$
\int_{I''_m} dx_1 \frac{1}{|DT(x_1)|^2}\int_{S^{(3)}_{m,\gamma}(Tx_1)}\frac{O(\tilde{\eps})}{ [1+\frac{D^2T(\hat{x}_2)}{DT(x_1)}|T_{x_1}^{-1}x_2-x_1|][1+\frac{D^2T(\hat{x}_3)}{DT(x_1)}|T_{x_1}^{-1}x_3-x_1|]}dx_2dx_3
$$
We now rewrite the first summand as
$$
\mathcal{I}_{1,m}:=\Le(S^{(3)}_{m,\gamma})\int_{I''_m} dx_1 \frac{\hat{h}_3(x_1,x_1,x_1)}{|DT(x_1)|^2}\frac{1}{\Le(S^{(3)}_{m,\gamma})}
$$
\begin{equation}\label{C1}
\int_{S^{(3)}_{m,\gamma}(Tx_1)}\frac{dx_2dx_3}{ [1+\frac{D^2T(\hat{x}_2)}{DT(x_1)}|T_{x_1}^{-1}x_2-x_1|][1+\frac{D^2T(\hat{x}_3)}{DT(x_1)}|T_{x_1}^{-1}x_3-x_1|]}
\end{equation}
where we have suppressed the dependence on $Tx_1$ in the Lebesgue measure of the external $S^{(3)}_{m,\gamma},$ which are independent of $Tx_1$ when $x_1\in I''_m,$ and the second summand  as
$$
\mathcal{I}_{2,m}:=\Le(S^{(3)}_{m,\gamma})\int_{I''_m}dx_1\frac{1}{|DT(x_1)|^2}\frac{1}{\Le(S^{(3)}_{m,\gamma})}$$$$\int_{S^{(3)}_{m,\gamma}(Tx_1)}\frac{O(\tilde{\eps})\ dx_2dx_3}{ [1+\frac{D^2T(\hat{x}_2)}{DT(x_1)}|T_{x_1}^{-1}x_2-x_1|][1+\frac{D^2T(\hat{x}_3)}{DT(x_1)}|T_{x_1}^{-1}x_3-x_1|]}
$$
\\ Using same arguments we have:
$$
\hat{\mu}\big(\mathcal{S}^{(3)}_{m})=\Le(S^{(3)}_{m})\int_{I'''_m} dx_1 \hat{h}_3(x_1,x_1,x_1)\frac{1}{\Le(S^{(3)}_{m})}\int_{S^{(3)}_{m}(x_1)} dx_2dx_3+
$$
\begin{equation}\label{C2}
\Le(S^{(3)}_{m})\int_{I'''_m}dx_1\frac{1}{\Le(S^{(3)}_{m})} \int_{S^{(3)}_{m}(x_1)}O(\tilde{\eps})  dx_2dx_3+{\bf I}(I\setminus I'''_m)=\mathcal{I}_{3,m}+\mathcal{I}_{4,m}+{\bf I}(I\setminus I'''_m)
\end{equation}
and with a lower bound for $\hat{\mu}\big(\mathcal{S}^{(3)}_{m})$ without the ${\bf I}(I\setminus I'''_m)$ term.
Hence we get
$$
\frac{\mathcal{I}_{1,m}+\mathcal{I}_{2,m}}{\mathcal{I}_{3,m}+\mathcal{I}_{4,m}+{\bf I}(I\setminus I'''_m)}\le \frac{\hat{\mu}(S^{(n)}_m\cap \hat{T}^{-1} S^{(n)}_m)}{\hat{\mu}(S^{(n)}_m)}\le\frac{\mathcal{I}_{1,m}+\mathcal{I}_{2,m}+{\bf I}(I\setminus I''_m)}{\mathcal{I}_{3,m}+\mathcal{I}_{4,m}}
$$
As in the proof of Proposition \ref{PROPO1}, we have that ${\bf I}(I\setminus I''_m)\le ||\hat{h}||_{\infty}\Le (S^{(3)}_{m,\gamma}) \Le(I\setminus I''_m)$ and ${\bf I}(I\setminus I'''_m)\le ||\hat{h}||_{\infty}\Le (S^{(3)}_{m}) \Le(I\setminus I'''_m).$
We can then factorize in the numerator and in the denominator the Lebesgue measures of
 the sets $S^{(3)}_{m,\gamma}$ and $S^{(3)}_{m}$ and remember that $\frac{\Le (S^{(3)}_{m,\gamma})}{\Le(S^{(3)}_{m})}=\frac{1}{(1-\gamma)^{2}}.$ After this factorization and when $m$ goes to infinity, the remaining part of  $\mathcal{I}_{1,m}$ converges to  $\int_Idx_1\frac{\hat{h}_3(x_1,x_1,x_1)}{|DT(x_1)|^2}$ by the dominated convergence theorem and the fact that  $|T^{-1}_{x_1}x_j-x_1|\le \nu_m, j=2,3,$  while the remaining part of $\mathcal{I}_{2,m}$ converges to an $O(\tilde{\eps})$ term. Still after the previous factorization,  the remaining part of   $\mathcal{I}_{3,m}$ goes to $\int_I dx_1 \hat{h}_3(x_1,x_1,x_1),$ while the remaining part of $\mathcal{I}_{4,m}$ goes to an $O(\tilde{\eps})$ term. The result then follows sending $\tilde{\eps}$ to zero.
\end{proof}
It is possible to relax the continuity assumption {\bf P7} on the density by asking a much weaker property. It seems to us that this condition is natural, and probably unavoidable,  in the sense that it controls the oscillations  of the density in the neighborhood  of the diagonal.
\begin{itemize}
\item {\bf P8} Let us suppose the density $\hat{h}$ is in $V_{1}(I^n)$ and moreover
    $$
    \hat{h}_D:=\sup_{0<\eps\le \eps_0}\frac{1}{\eps}\int \text{osc}(\hat{h}, B_{\eps}(x,\cdots,x))) dx<\infty.
    $$
\end{itemize}
\begin{prop}\label{proppp}
Let us suppose that our CML satisfies properties {\bf P1}-{\bf P5} and {\bf P8} on the Banach space $\mathcal{B}=V_{1}(I^n)$ with $\lambda=\inf|DT|^{-1}<1-\gamma,$  then the statement in Proposition \ref{propp} holds.
\end{prop}
\begin{proof}
The proof follows the line of Proposition \ref{propp}, with an essential change when we compare the density in the neighborhood of the point $(x_1,x_1,x_1)$. In fact, we can now write
$$
|\hat{h}_3(x_1, T_{x_1}^{-1}x_2, T_{x_1}^{-1}x_3)-\hat{h}_3(x_1,x_1,x_1)|\le \text{osc}(\hat{h}_3, B_{\nu_m}(x_1,\cdots,x_1)).$$ An quick inspection of the previous proof shows immediately that the integral
$\int_I dx_1\  O(\tilde{\epsilon})$ will be now replaced with $\int_I dx_1 \ \text{osc}(\hat{h}_3, B_{\nu_m}(x_1,\cdots,x_1)),$ and this last integral is bounded by $\hat{h}_D \nu_m,$ which goes to zero when $m$ tends to infinity.
\end{proof}
\begin{cor}
As a consequence of  Propositions \ref{propp} and \ref{proppp},  the extremal index $\theta_n$ for maps satisfying {\bf P0} too,  is given by
\begin{equation}\label{EEIII}
\theta_n=1-\frac{1}{(1-\gamma)^{n-1}}\frac{\int_I \frac{\hat{h}_n(x,\cdots, x)}{|DT(x)|^{n-1}}dx }{\int_I \hat{h}_n(x,\cdots,x)dx}.
\end{equation}
and it will converge to $1$ when $n\rightarrow \infty.$
\end{cor}
\subsection{Random perturbations} As for localization, we expect that the extremal index be one when we keep $n$ fixed and we add  noise to the system. In the paper \cite{AFV} we extended the SA to randomly perturbed dynamical systems, mostly with additive noise.  Even if we assume properties ({\bf P1})-({\bf P6}) on some Banach space $\B$, there will be a new  difficulty related to the computation of the quantities $q_k$ in (\ref{Q3}) in the random setting. Such a computation as it was done in  Proposition 5.3   in \cite{AFV} strongly relies on the fact that the observable becomes infinite in a single point, the center  of a ball: we do not know how to adapt it in the neighborhood of the diagonal $\Delta.$ We will present nevertheless numerical evidences in Section 7 that in presence of noise the EI is $1.$
\section{Distribution of the number of successive  visits}
We anticipated in the introduction that once the synchronization is turned on for the first time, it cannot last since almost every orbit is recurrent. However  the orbit $\hat{T}^n(\overline{x}_0)$ will visit for almost every point $\overline{x}_0$ infinitely often  the neighborhood of the diagonal. We could therefore  expect that the exponential law $e^{-\tau}$ given by the EVT describes the time between successive events in a Poisson process. To formalize this, let us take a neighborhood $S^{(n)}_{\varsigma}$ of the diagonal $\Delta$ with accuracy $a_c=\varsigma$  and introduce the following quantity (remember that the map $\hat{T}$ and the measure $\hat{\mu}$ depend on $n$ too):
$$
N^{(n)}_{\varsigma}(t)=\sum_{l=1}^{\left \lfloor{\frac{t}{\hat{\mu}(S^{(n)}_{\varsigma})}}\right \rfloor}{\bf 1}_{S^{(n)}_{\varsigma}}(\hat{T}^l(\ox)),
$$
where $\left \lfloor{\cdot}\right \rfloor$ is the floor function, and consider the following distribution
$$
\mathcal{N}(n,\varsigma,t,k):=\hat{\mu}(N^{(n)}_{\varsigma}(t)=k)
$$
If the target set was a ball of radius $\varsigma$ around a generic point $\overline{z}$ or a dynamical cylinder set converging to this point,  one can prove under the  mixing assumptions of our  paper,    that in the limit of vanishing radius or infinite length for the cylinder,   $\mathcal{N}(n,\varsigma,t,k)$ converges to the Poisson distribution $\frac{t^ke^{-t}}{k!},$ see for instance \cite{V1},\cite{V2}. Instead if we take the target point $z$ periodic of minimal period $q$, one get the so-called compound Poisson distribution, see  \cite{HV} and \cite{AJFM3}, which in our situation reads, for $k\ge 1:$
\begin{equation}\label{PO}
\mathcal{N}(n,\varsigma,t,k)= e^{-t(1-p)}\sum_{j=0}^kp^{k-j}(1-p)^{j+1}\frac{t^j(1-p)^j}{j!}\binom{k-1}{j-1}
\end{equation}
where
\begin{equation}\label{pp}
p=\frac{1}{|\det(D_{\overline{z}}\hat{T}^q)|}
\end{equation}

\begin{rmk}
We do not dispose for the moment of analogous formulas when a ball is replaced by a strip along our diagonal set $\Delta.$ To the best of our knowledge the only known  result  is in dimension $2$ for the {\em uncoupled} systems given by the direct product of two piece-wise expanding and smooth
maps of the circle, see \cite{CC}, and it is consistent with our results. Nevertheless a few preliminary considerations\footnote{At this regard see also the discussion in the last part of Section 7.} seem to indicate  that the compound distribution (\ref{PO}) still holds with $p$ in (\ref{pp}) replaced by $1-\theta_n$ in (\ref{EEIII}),  and more generally with the EI given by formulas (\ref{Q2}), (\ref{Q3}), with the quantities $q_{k,l}$ given by the right hand side of (\ref{PIVA}) when the transfer operator is not available. In particular one should recover a pure Poisson distribution when the size $n$ of the lattice tends to infinity.
\end{rmk}
\begin{example}(Ex. (\ref{EX}) revisited)
\begin{itemize}
\item Suppose we consider as in the example (\ref{EX}), Ex. 2,  $n=100$ particles living in the unit interval and take the accuracy $\varsigma=0.01.$ With that value of $n$ and taking the coupling $\gamma$ sufficiently small, we could consider that the previous number of visits $N_{\varsigma}^{(n)}(t)$ follows a Poisson distribution. Since the probability of entering the neighborhood of the diagonal is of order $100^{-100},$
the
probability to observe exactly $5$ synchronization events during $m$ iterations of the lattice is maximal for $m=5\ 100^{100}$ and is of  order  $18\%.$
\item If instead we consider Ex. 1 with $3$ particles and the same accuracy, the probability to observe $5$ synchronizations is maximal after  $50\ 000$ iterations and it is again of order  $18\%.$
    \end{itemize}
\end{example}
\begin{com}
In the case of large $n$ the extremely high number of iterations needed to get synchronization or a given number of successive  synchronizations   could surprise. One reason is surely due to the fact that we considered lattices which are globally coupled and we looked at global synchronization. It would be interesting, and it will be the objects of future investigations,  to explore CML where only the nearest-neighbors of a given site contribute to the coupling term ({\em diffusive coupling}), and also synchronization of the closest neighbors. About the latter we will give a few preliminary numerical results in the next section.
\end{com}
\section{Extensions and Numerical computations}

The goal of numerical computations  will be to show that in the situations considered above we have effective convergence toward an extreme value law and moreover the extremal index satisfies the behavior we predicted theoretically. We will be mostly interested in synchronization, since for localization we have plenty  of analytic results. But there is one aspect where the comparison with localization is particularly useful. In order to explain that, we first have to introduce a new observable to depict a different kind of synchronization.
\subsection{Local synchronization}
 Up to now synchronization  was defined by asking that all the components of the evolutionary state become close to each other with a given accuracy $a_c$. We could ask instead that each component synchronize only with the  close neighbors. This is done by introducing the following observable

\begin{equation}\label{S2}
 \Theta(\overline{x}):=-\log\{\max|x_i-x_j|, i\neq j: j=i\pm1\}
 \end{equation}
    (of course on the extreme points of the period of the lattice, $j$  will take only one value).
     We could generalize to more than one neighbor $j=i\pm 2, \pm3,$ etc., but we limit ourselves here to the case $\pm1.$ It is not immediately obvious to have a geometrical description of the set that the orbit will visit for the first time (and therefore to give analytic results  in terms of the EI), although the ``physical" interpretation will be the same, namely we get the probability that the lattice will have for the first time and after a given number of iterations $m$, all the components synchronized with the close neighbors and with a given accuracy $a_c$. We call this {\em local synchronization}, to distinguish   from the {\em global synchronization} described in the preceding sections. It seems intuitive  from a physical point of view, that for $m$ large enough and for a given accuracy $a_c$, the probability to get local synchronization for the first time (from now on we write it as $\mathbb{P}_1(\cdot)$ for the different cases), is larger than that to get global synchronization, $\mathbb{P}_1(\text{glob. sync.})\le \mathbb{P}_1(\text{loc. sync.}),$ and this will be confirmed by the numerical simulation as we will see in a moment. On the other hand as soon as the global synchronization occurs, all the components of the lattice will be aligned in a narrow strip around all of them, and this is close  to localization. Therefore we will expect that
 the probability to get localization is larger than the probability of global synchronization. This is also confirmed by an easy  application of the theory.   Suppose we fix $m$ and the accuracy $a_c$; we have also fixed $n$. By supposing a pure exponential law for the asymptotic distribution of the maximum, we have
    \begin{itemize}
    \item For localization: $a_c\sim (\frac{\tau}{m})^{\frac1n},$ which gives $\mathbb{P}_1 (\text{local.})\sim e^{-\tau}\sim e^{-ma_c^n}.$
        \item For global synchronization: $a_c\sim (\frac{\tau}{m})^{\frac{1}{n-1}},$ which gives $\mathbb{P}_1 (\text{glob. sync.})\sim e^{-\tau}\sim e^{-ma_c^{n-1}}.$
    \end{itemize}
    We see that $\mathbb{P}_1 (\text{glob. sync.})\le \mathbb{P}_1 (\text{local.}).$\\
\subsection{Blocks of synchronization}
    The observable (\ref{S2}) could be modified further by introducing a  new one which we are going to define. Let us first construct $N$ blocks of $L$ successive integer indices: $B_q:=\{i_q, \cdots, i_q+L\}$ and take these blocks disjoint and possibly scattered along the lattice.  Then we define:
    $$
    \Upsilon(\ox):=-\log\{\max|x_i-x_j|, i\neq j: (i,j)\in B_q, q=1,\cdots, N \}.
    $$
The distribution of the maximum of this observable will give us the probability that the particles in the $N$ blocks will synchronise for the first time with a given accuracy. On the other hand we do not require any synchronization of the particles outside those blocks. If such a limiting distribution would exist, it could be consistent with the appearance of {\em chimeras} in chains of coupled particles, namely patterns of synchronized sets which emerge as a consequence of the self-organization of the entire lattice, see e.g. \cite{NPRL}. If our claim would be confirmed, such a self-organization would be  another statistical property of chaotic systems with several degrees of freedom. \\

\subsection{Simulations}
Let us now analyze the results of numerical procedure. The experiment performed is the following: we consider the one-dimensional map  $T$ in (\ref{TM}) as
$$T(x)=3x \mod 1. $$
Once we have constructed the CML $\hat{T}$ we will perturb it with additive noise:
$$\T_{\o}(\overline{x})_i=\hat{T}(\ox)_i +\epsilon\omega_i\mod 1 $$ where $\epsilon$ is here the noise intensity and $\o$ with components $\omega_i$ is a random variable drawn from a uniform distribution between -0.5 and +0.5. The stationary measure
for such a map will be $L^1$ close to that for $\gamma=0$ which is the direct product of  the uniform Lebesgue measures on the unit circle for each component and this independently
of the value of $\epsilon.$ Let us  notice that we are considering now  a one-dimensional map on the circle. This is not a restriction to our previous considerations and moreover it allows us to define correctly the additive noise. Numerically we produce trajectories of $10^4$ iterations for $\gamma<2/3$ and 0.02 increments. The range $3<n<53$ is analyzed. We consider the two observables $\psi,$ see (\ref{S1}) and $\Theta,$  see (\ref{S2}), corresponding  to global and local synchronization cases respectively and in the following we will refer to them as the {\em global and local cases}.  We analyze also the role of small noise $\epsilon=10^{-4}$ and moderate noise $\epsilon=10^{-2}$.\\

We first assess the convergence of the  maxima of $\psi$ and $\Theta$ to the Gumbel law by analyzing the tail index $\xi,$ see Section 3.  Here we chose to consider the complementary approach to the block-maxima selection, i.e. the peak over threshold. The two approaches are equivalent in chaotic systems as shown in \cite{ERDS}. The maxima of the observable are defined as the exceedances over the 0.98 quantile of $\psi$ and $\Theta$ distributions. If a good convergence towards the Gumbel law is reached, then $\xi \simeq 0$. The values of $\xi$  as a function of $\gamma$ and $n$ are reported in Figure \ref{CLM_csi}. A maximum likelihood estimator has been used for computation. The left panels show the global case $\psi$ while the  local case  $\Theta$ is reported on the right. From top to bottom we switch on the noise.  In general, the convergence towards the Gumbel law is satisfactory although some differences exist between global and local cases. For the global case the convergence is slower as the global synchronization event  is more rare then the local one.

 Moreover, the quality of the fits is lower when $n$ and $\gamma$ are larger. The addition of noise helps the convergence to the Gumbel law as for the systems analyzed  in \cite{ERDS}.\\

We now study the implications of global and local synchronization  on the  extremal index $\theta$. For the analysis presented in this paper, we adopt the  estimator  by
S\"uveges (see the book \cite{ERDS} for explanation and to retrieve the codes for the computation). For fixed
quantile $q$, S\"uveges' estimator reads:

$$ \theta=\frac{\sum_{i}^{N_c} (1-q)S_i +N+N_c- \sqrt{ \left(\sum_{i}^{N_c} (1-q)S_i  +N+N_c\right)^2  -8Nc\sum_i^{N_c} (1-q) S_i  }}{2\sum_{i}^{N_c} (1-q) S_i },$$

where $N$ is the number of recurrences above the chosen quantile, $N_c$ is the number of observations which form a
cluster of at least two consecutive recurrences, and $S_i$ the length
of each cluster $i$. From the numerical point of view, this estimator is the expected value of the compound distribution $\mathcal{N}(n, \varsigma, t, k)$ with $S_i$ being the empirical equivalent of the quantity $N_{\varsigma}^{(n)}(x)$.

We begin by checking the theoretical results predicted in Remark 5.4: for the 3$x$ mod 1 map, $\theta_n$ can now be estimated by taking the trace of the density on the diagonal reasonably of order $1$  in (\ref{EEIII}), so that in dimension 2:
$
\theta_2\sim 1- \frac{1}{(1-\gamma)}\frac{1}{3}
$
and in dimension $3$:
$
\theta_3\sim 1- \frac{1}{(1-\gamma)^2}\frac{1}{9}.
$

The comparison between the theoretical curves and the numerical computations are shown in Figure  \ref{d32}. For each case $n=2,3$ and $\gamma<2/3$ we produce 10 simulations of the map consisting of $10^4$ iterations and we estimate the extremal index as a function of $\gamma$. The numerical estimates indeed match the theoretical curves (bold magenta lines).\\

We now check the asymptotic formula for large $n$ and still with the same assumption on the trace of the density, namely $\theta_n\sim  1-(\frac{\lambda}{1-\gamma})^{n-1},$ with $\lambda=1/3.$ For each $3<n<53$ and $\gamma<2/3$ we perform one simulation of the deterministic 3$x$ mod 1 map and compare the obtained extremal index $\theta_n$ with the previous asymptotic formula. Results are shown in Figure \ref{CLM_t}. There is indeed very good agreement between our asymptotic and numerical results. The largest divergence is obtained for $\gamma\simeq 2/3$ which correspond to the limit value for the map.\\

We then perform a  numerical analysis of the extremal index in the cases not covered by the theory, namely for the observable $\Theta.$  The results are presented in Figure  \ref{CLM}. The top-left panel is repeated for convenience and show the global case results. The latter show that global and local cases are substantially different. For the global, the synchronization depends on both $n$ and $\gamma$: in particular, it is easier to synchronize systems with $n$ small because the probability of finding all the particles in the same state decreases quickly with $n$.  On the other hand, in the local  case the extremal index $\theta_n$ is substantially independent of $n$. In fact, whether $n$ is small or large, the particle \textit{sees} only the nearest neighbors for synchronization so that it is insensitive to the size of the lattice. The only dependence left is in $\gamma$: in particular, for all $n$, we see the dependence is compatible with the case $n=2$ of the global coupling case: $\theta_2\sim 1- \frac{1}{(1-\gamma)}\frac{1}{3}$. The addition of the noise destroys clusters as observed in \cite{AFV}. Qualitatively, the structure of the extremal index is quite robust with respect to small perturbations.   To fully destroy the clusters, large intensity of the noise are needed. The results for $\epsilon=10^{-4}$ also demonstrate that our results are stochastically stable because one   recovers the deterministic structure of the extremal index for low noise values.\\

Although the numerical estimates of the extremal index are done by computing the expected values of the compound Poisson distribution (S\"uveges' estimator), we can also check that the waiting times $q_{k,\varsigma}$\footnote{We now index this quantity with the size $\varsigma \rightarrow 0$ of the neighborhood of the diagonal.}  defined in (\ref{PIVA}),  between consecutive entrances in the neighborhood of the diagonal with accuracy $\varsigma,$  provides the same information. Actually this is what we get for recurrence in balls as we discussed above, see \cite{HV} and \cite{AJFM3}.   Therefore we give some examples of time series of  $\psi$ and $\Theta$ in Figures \ref{global_Poisson} and \ref{local_Poisson} respectively. The noise increases from top to bottom. The histograms of the waiting times in cluster are normalized to sum-up to $1$ (empirical probability density function EPDF) and are in $y$-$\log$ scale. No clustering corresponds to  an exponential law (sequence of linearly decreasing boxes in $\log$ scale), whereas the clustering case is characterized by an higher EPDF for lower waiting times. As one can see from the deterministic cases, the higher the EPDF for short waiting times, the lower $\theta$. Effectively the fraction of waiting times equals to $1$ which exceed the standard exponential law is exactly the extremal index  $\theta$. We stress again that although we cannot demonstrate this identity theoretically, the numerical evidence suggests that one can use directly $q_{k,\varsigma}$ as defined in (\ref{PIVA}),  for the estimation of the extremal index $\theta$.\\
\section {Appendices}
\subsection{Appendix 1: proof of (\ref{EDC})}
The argument is the following. The quantity we are interested in is bounded by $\int_{I^n}d\ox |\prod_{i\neq j}{\bf 1}_{\{|x_i-x_j|<\nu_l+\eps\}}(\ox)-\prod_{i\neq j}{\bf 1}_{\{|x_i-x_j|<\nu_l-\eps\}}(\ox)|.$ If at least one factor in  the first  product is zero, the same is true for the second product, so we will suppose that all the factors in the first product are $1$. Therefore the difference of the two products will be maximum if at least one factor in the second product is zero. There will be at most $\sum_{k=1}^n \binom{k}{n}$ such possibilities. We now proceed with a very rough bound. Each term in $\binom{k}{n},$ with $1\le k\le (n-1)$ contributes with $k$ measures of values $4^k\eps^k$ and with $(n-k)$ measures of values $4^{n-k}\nu_l^{4-k},$ having chosen $\eps<\nu_l$. When $k=n$ we simply write $\eps^n<\eps^{n-1}\nu_l.$ In conclusion, we bound the quantity we are interested in by $\eps \nu_l C_n,$ with $C_n=4^n\sum_{k=1}^n  \binom{k}{n}.$
\subsection{Appendix 2: proof of (\ref{CAR})}
 Take for simplicity $n=2.$ There is in fact dependence of the two sets on $x_1$ since they intersect $I^2\ni (x_2, x_3)$ and as a consequence their measure will  depend on the location of $x_1.$ It  will  therefore be enough to evaluate  the external integrals in $x_1$ on a even smaller domain  $I''_m\subset I'_m$ and on $I'''_m$ in the   denominator, in such a way they will not  contain a (disconnected) neighborhood $\mathcal{U}$ of $0$ and $1$ and its preimages $T^{-1}\mathcal{U}.$   As a consequence, we can keep the full amount of the area of the two sets $S^{(2)}_{m,\gamma}(Tx_1)$ and $S^{(2)}_{m}(x_1),$ which from now on we simply write as $\Le (S^{(2)}_{m,\gamma})$ and $\Le(S^{(2)}_{m}).$ Clearly the difference between the integrals over $I$ and $I''_m, I'''_m$ will converge again  to zero when $m\rightarrow \infty.$
 About the other issue: write $S^{(n)}_{m,\gamma}(Tx_1)$ as the integral of obvious characteristic functions in the variables $x_2, \dots, x_n.$ Then make the change of variables:  $x_k'=x_k (1-\gamma)+\gamma T(x_1)$, in this way we get the measure of $S^{(n)}_{m}(x_1)$  multiplied by $(1-\gamma)^{1-n}.$
\section*{Acknowledgments}

SV and PG were supported by  the MATH AM-Sud Project “Physeco”. SV was supported  by  the Leverhulme Trust  for support thorough the Network Grant IN-2014-021  and by the project APEX “Syst\`emes dynamiques: Probabilit\'es et Approximation Diophantienne
PAD” funded by the R\'eegion PACA (France). DF was partially supported by the ERC grant A2C2 (No. 338965). PG thanks FONDECYT project 1171427.  SV warmly thanks J. M. Freitas, P. Giulietti, N. Haydn and B. Saussol  for illuminating discussions.
We finally thank the referees for the careful reading of the paper which helped us to greatly improve it.

\clearpage

\begin{figure}[t]
\includegraphics[width=0.97\textwidth]{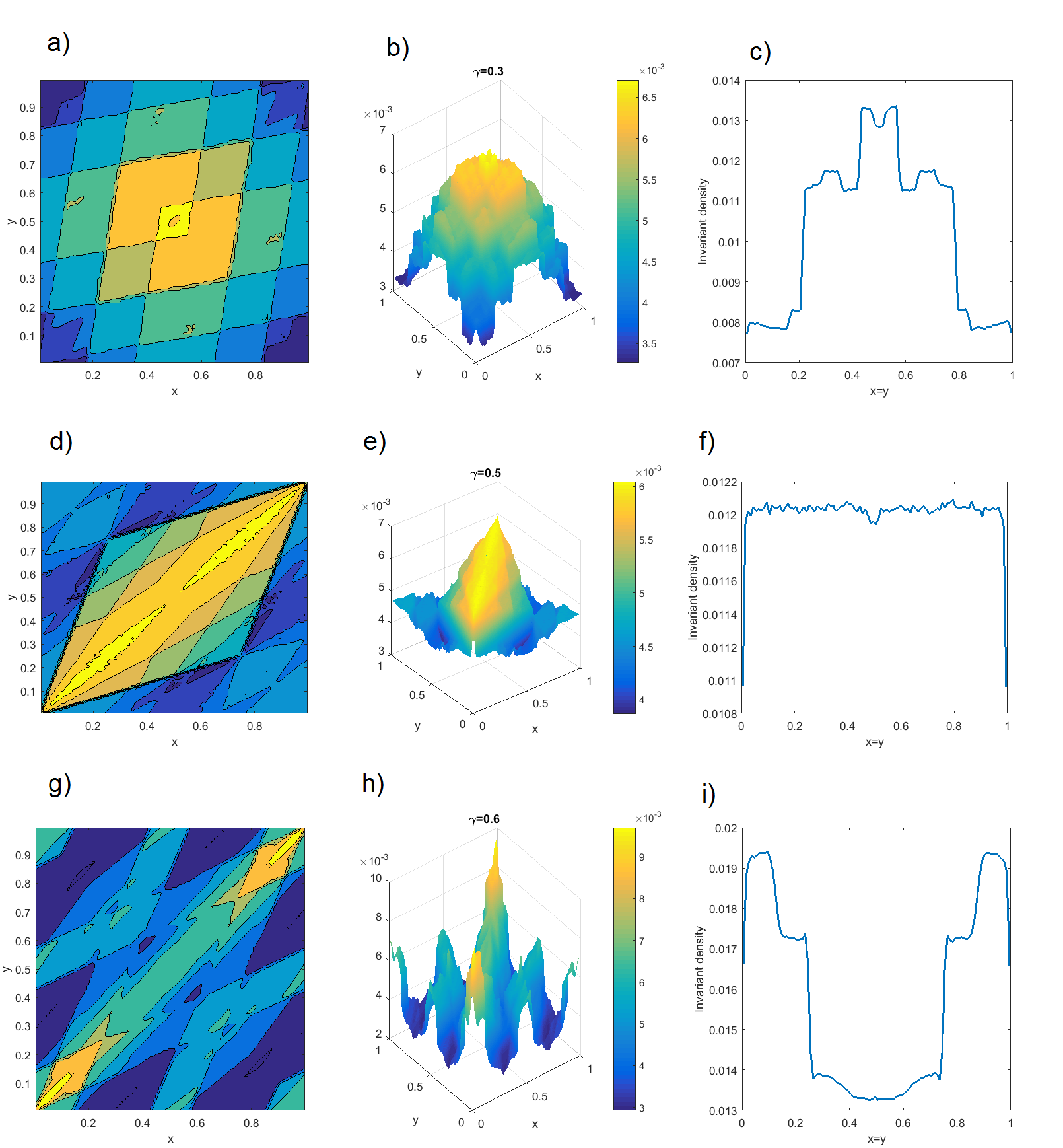}
\caption{Invariant density for the map \ref{EX0} with $n=2$ for $\gamma=0.3$ (a,b,c),  $\gamma=0.5$ (d,e,f), $\gamma=0.6$ (g,h,i). The plots show the density in colorscale (a,d,g) with a view from the top and (b,e,h) for a three dimensional view. The plots (c,f,i) show the behavior of the map on the diagonal. The figure is obtained by averaging the density over 300 realizations each consisting of $10^7$ iterations of the trajectory.}
\label{dens1}
\end{figure}

\begin{figure}[t]
\includegraphics[width=0.97\textwidth]{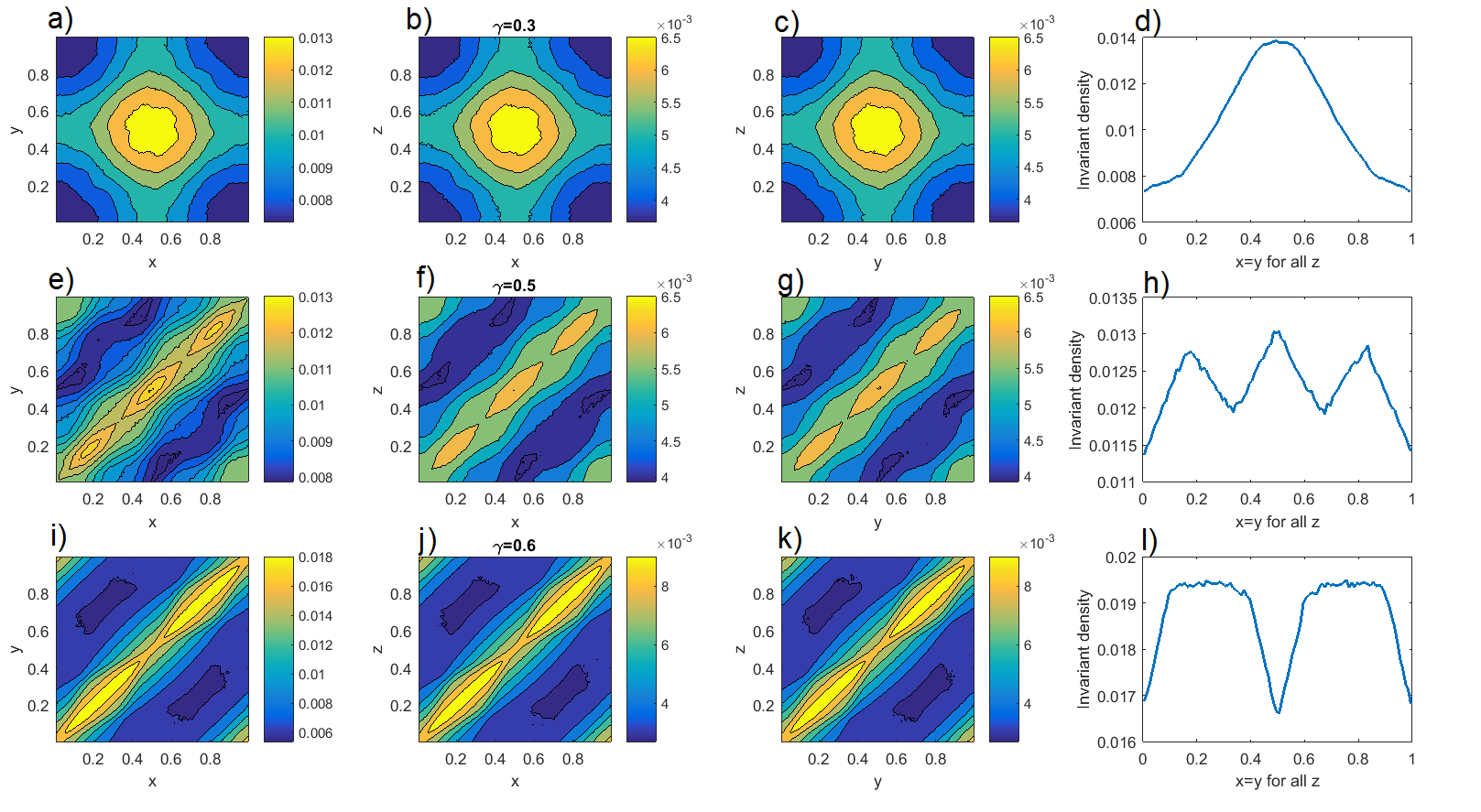}
\caption{Invariant density for the map \ref{EX0} with $n=3$ for $\gamma=0.3$ (a,b,c,d),  $\gamma=0.5$ (e,f,g,h), $\gamma=0.6$ (i,j,k,l). The plots show the density in colorscale (a-c,e-g,i-k). The plots (d,h,l) show the behavior of the map on the diagonal $x=y$. The figure is obtained by averaging the density over 300 realizations each consisting of $10^6$ iterations of the trajectory.}
\label{dens}
\end{figure}

\begin{figure}[t]
\includegraphics[width=0.47\textwidth]{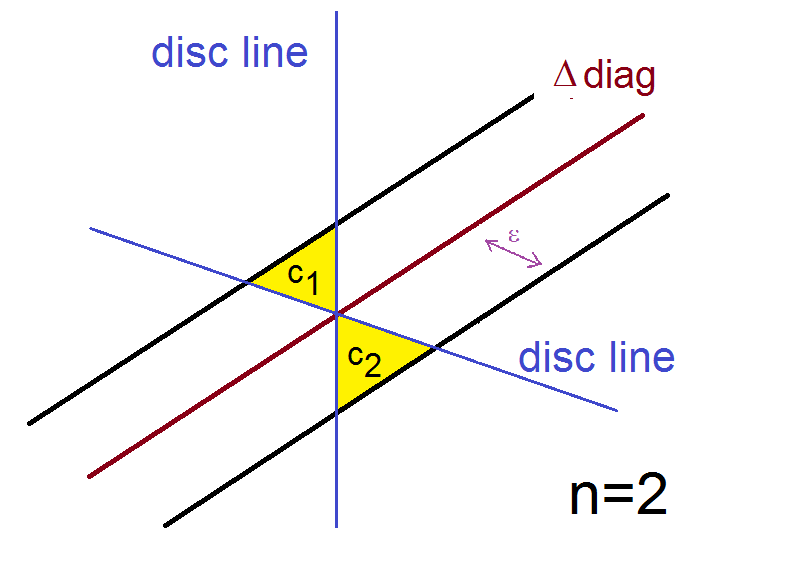}\includegraphics[width=0.47\textwidth]{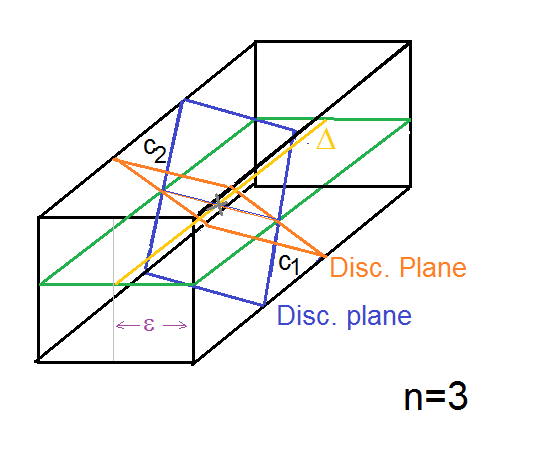}
\caption{Crossing of the discontinuity line, $n=2$ (left)  and surface, $n=3$ (right),  of the neighborhood of the diagonal $\Delta,$ for Property {\bf P01}. $C_1$ and $C_2$: triangular (left) and pyramidal (right) regions belonging respectively to $F^c_{d, \varepsilon, 2}$ and $F^c_{d, \varepsilon, 3}$. We remove the shaded regions on the left and the pyramidal regions $C_1$ and $C_2$ on the right.}
\label{DIPO}
\end{figure}

\begin{figure}[t]
\includegraphics[width=0.97\textwidth]{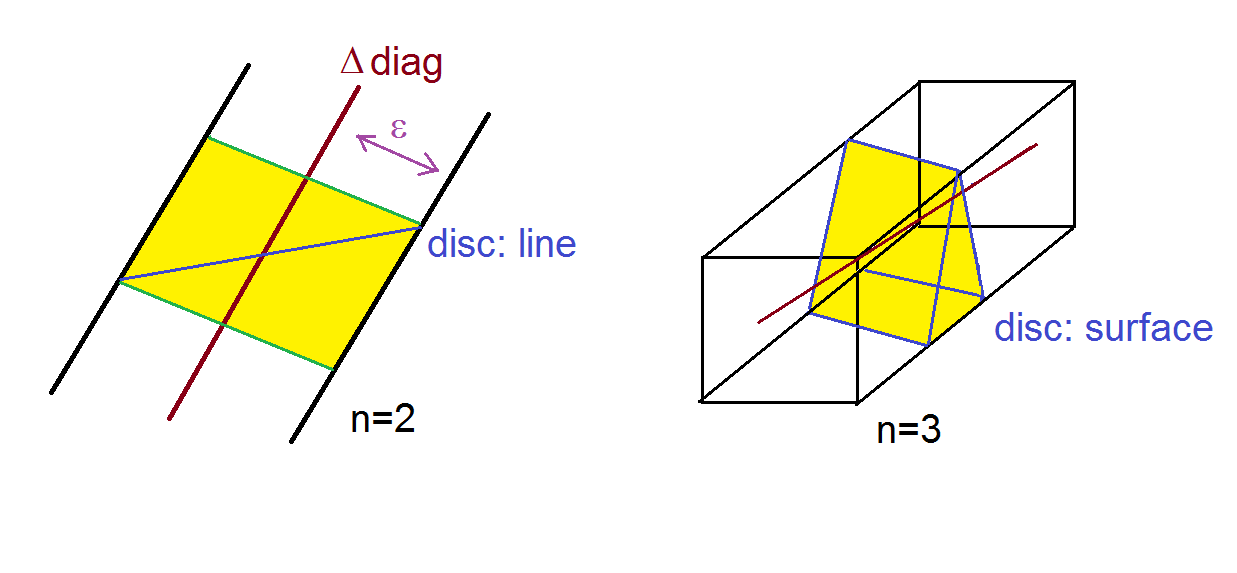}
\caption{Crossing of the discontinuity line, $n=2$ (left)  and surface, $n=3$ (right),  of the neighborhood of the diagonal $\Delta,$ for  Property {\bf P02}. We remove the shaded regions.}
\label{DIS}
\end{figure}

\begin{figure}[t]
\includegraphics[width=0.97\textwidth]{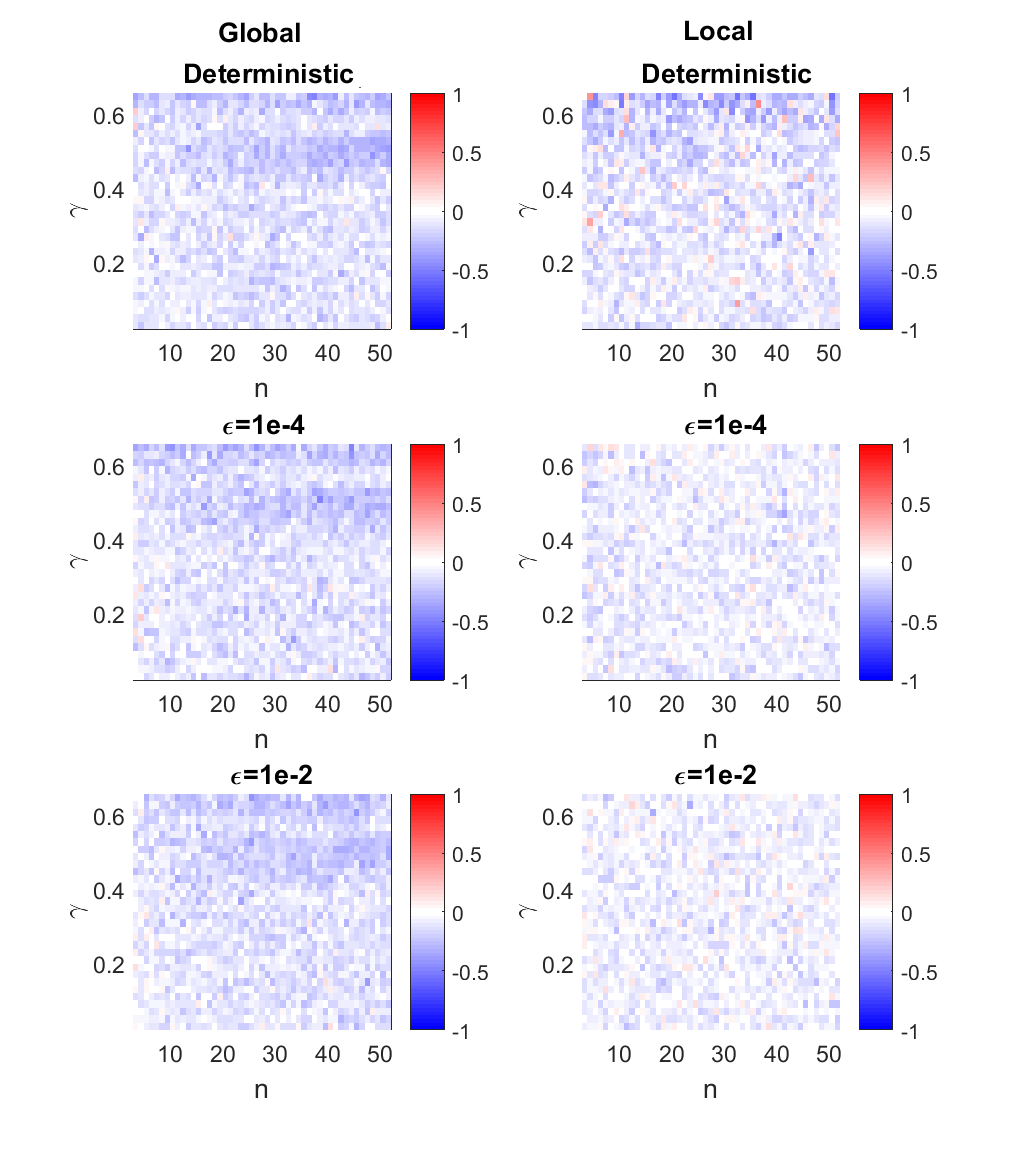}
\caption{Shape parameter $\xi$ of the Generalized Pareto distribution as a function of the number of variables $n$ and the coupling parameter $\gamma$. Left: global case $\psi$. Right: local case $\Theta$. From top to bottom: Deterministic, additive noise with intensity $\epsilon=10^{-4}$, additive noise with intensity $\epsilon=10^{-2}$.}
\label{CLM_csi}
\end{figure}

\begin{figure}[t]
\includegraphics[width=0.47\textwidth]{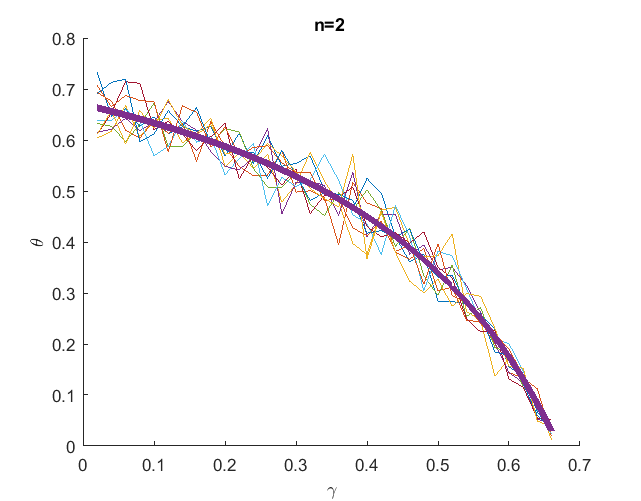}\includegraphics[width=0.47\textwidth]{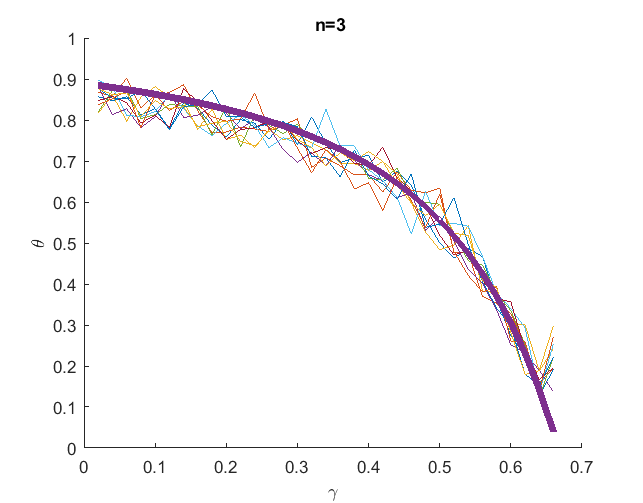}
\caption{Extremal index $\theta$ of the Generalized Pareto distribution as a function of the  coupling parameter $\gamma$. Thin lines indicate estimates for $10$ different realization of the maps $3x$-mod$ 1$. Bold magenta lines indicate the expected theoretical values. Left: $n=2$, Right: $n=3$.}
\label{d32}
\end{figure}

\begin{figure}[t]
\includegraphics[width=0.67\textwidth]{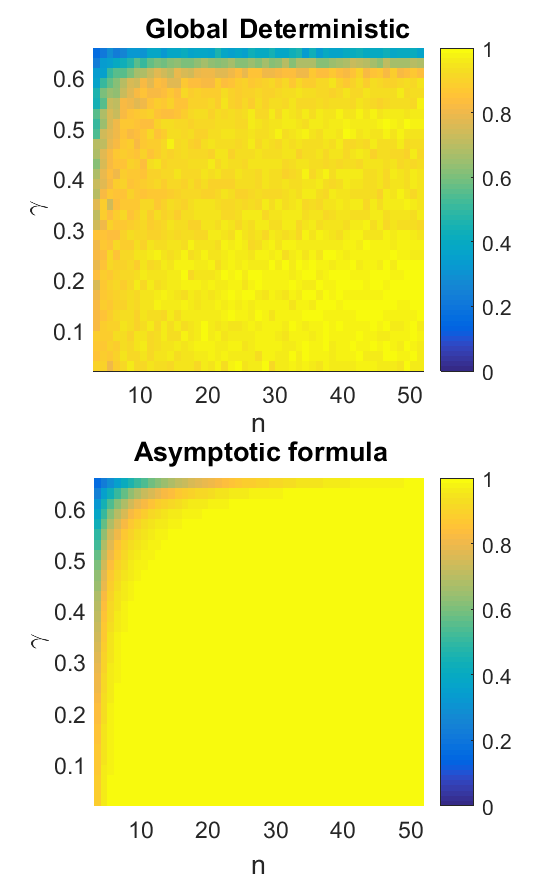}
\caption{Extremal index $\theta$ as a function of the number of variables $n$ and the coupling parameter $\gamma$. Top:  global case $\psi$. Bottom: theoretical asymptotic formula.}
\label{CLM_t}
\end{figure}

\begin{figure}[t]
\includegraphics[width=0.97\textwidth]{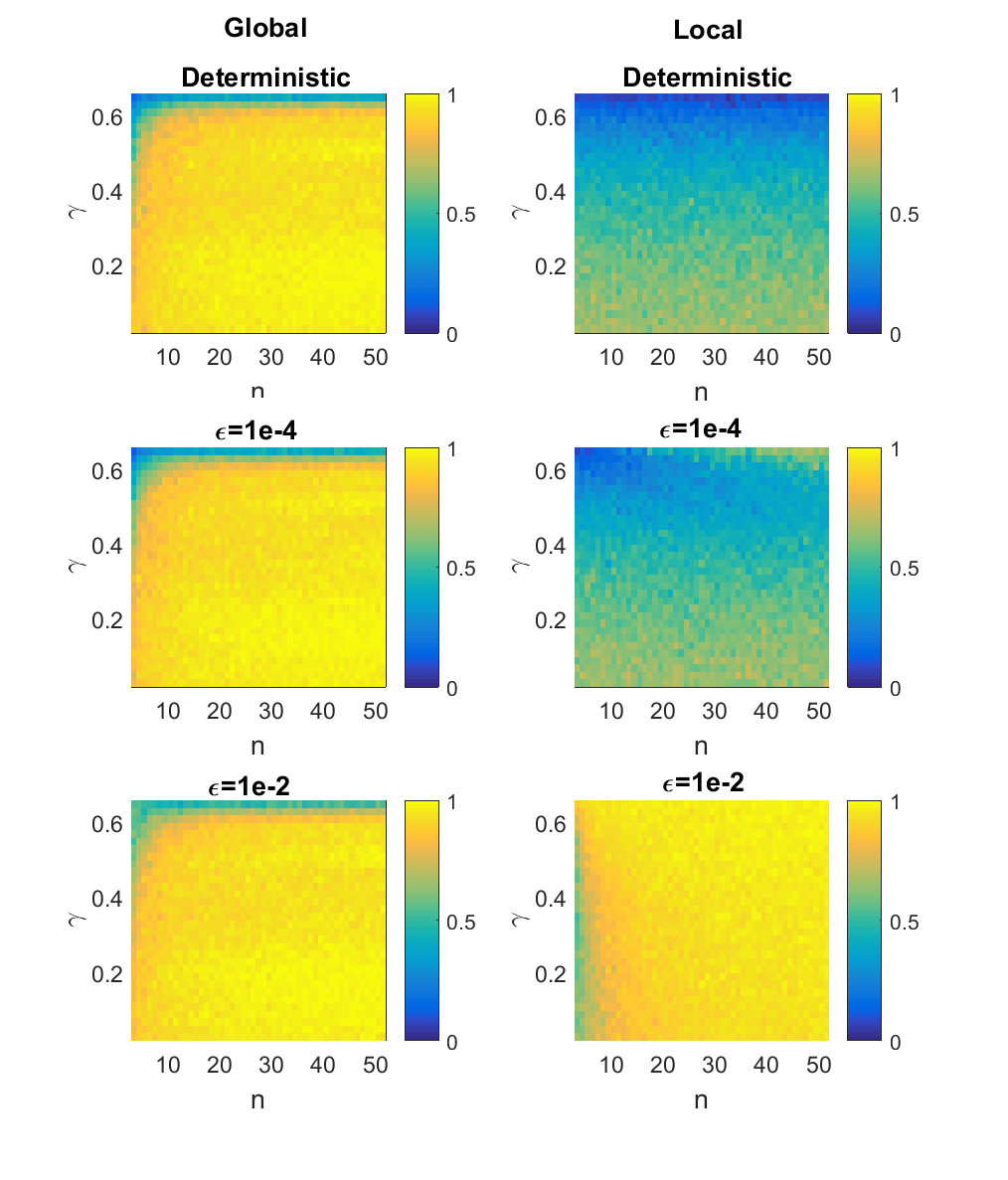}
\caption{Extremal index $\theta$ as a function of the number of variables $n$ and the coupling parameter $\gamma$. Left:  global case $\psi$. Right: local case  $\Theta$. From top to bottom: Deterministic, additive noise with intensity $\epsilon=10^{-4}$, additive noise with intensity $\epsilon=10^{-2}$.}
\label{CLM}
\end{figure}

\begin{figure}[t]
\includegraphics[width=0.97\textwidth]{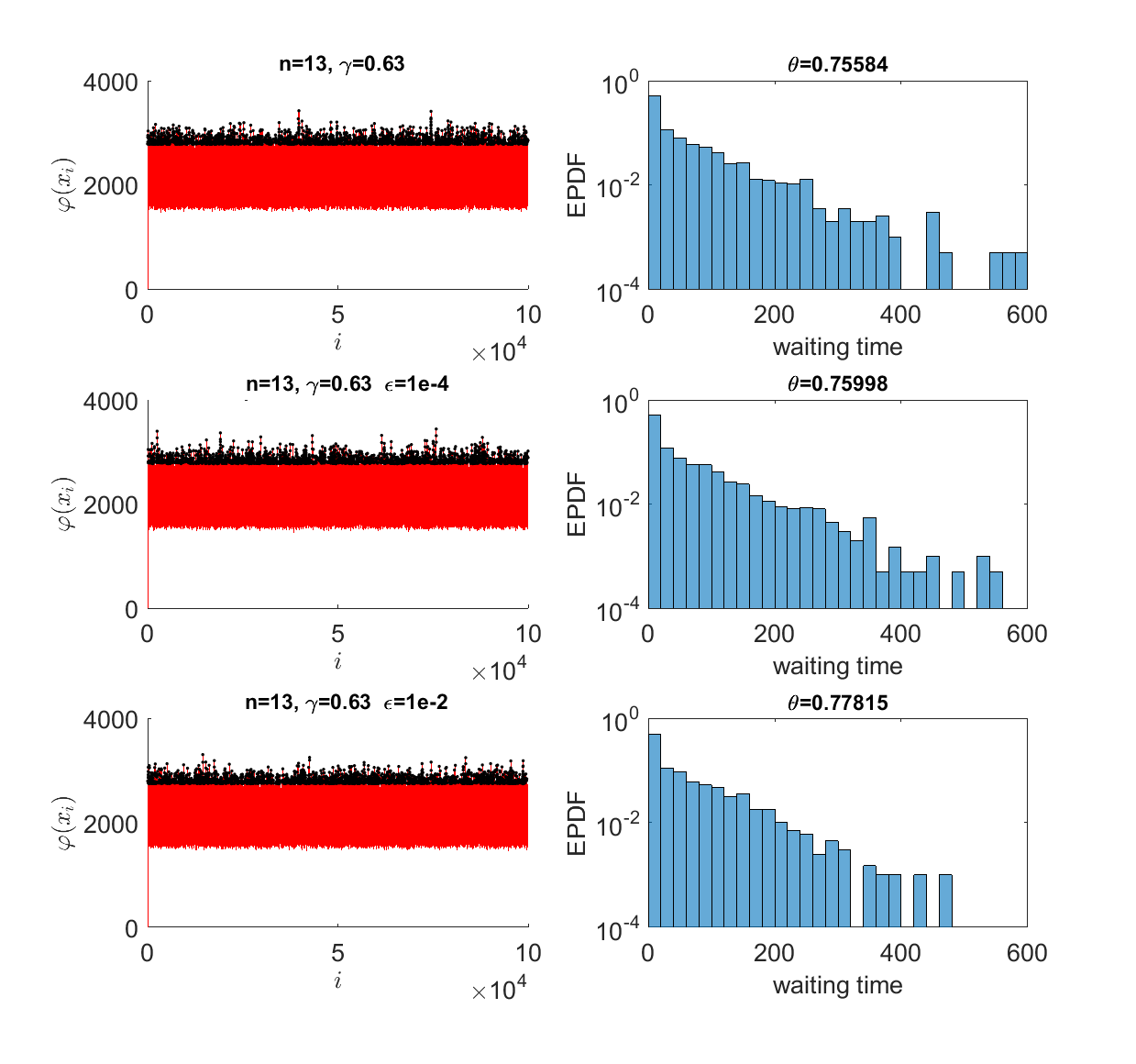}
\caption{Example of global case. Right: $\psi$  time series (red) and exceedances (black). Left: empirical probability distribution (EPDF) of waiting time in the clusters. From top to bottom: Deterministic, additive noise with intensity $\epsilon=10^{-4}$, additive noise with intensity $\epsilon=10^{-2}$.}
\label{global_Poisson}
\end{figure}

\begin{figure}[t]
\includegraphics[width=0.97\textwidth]{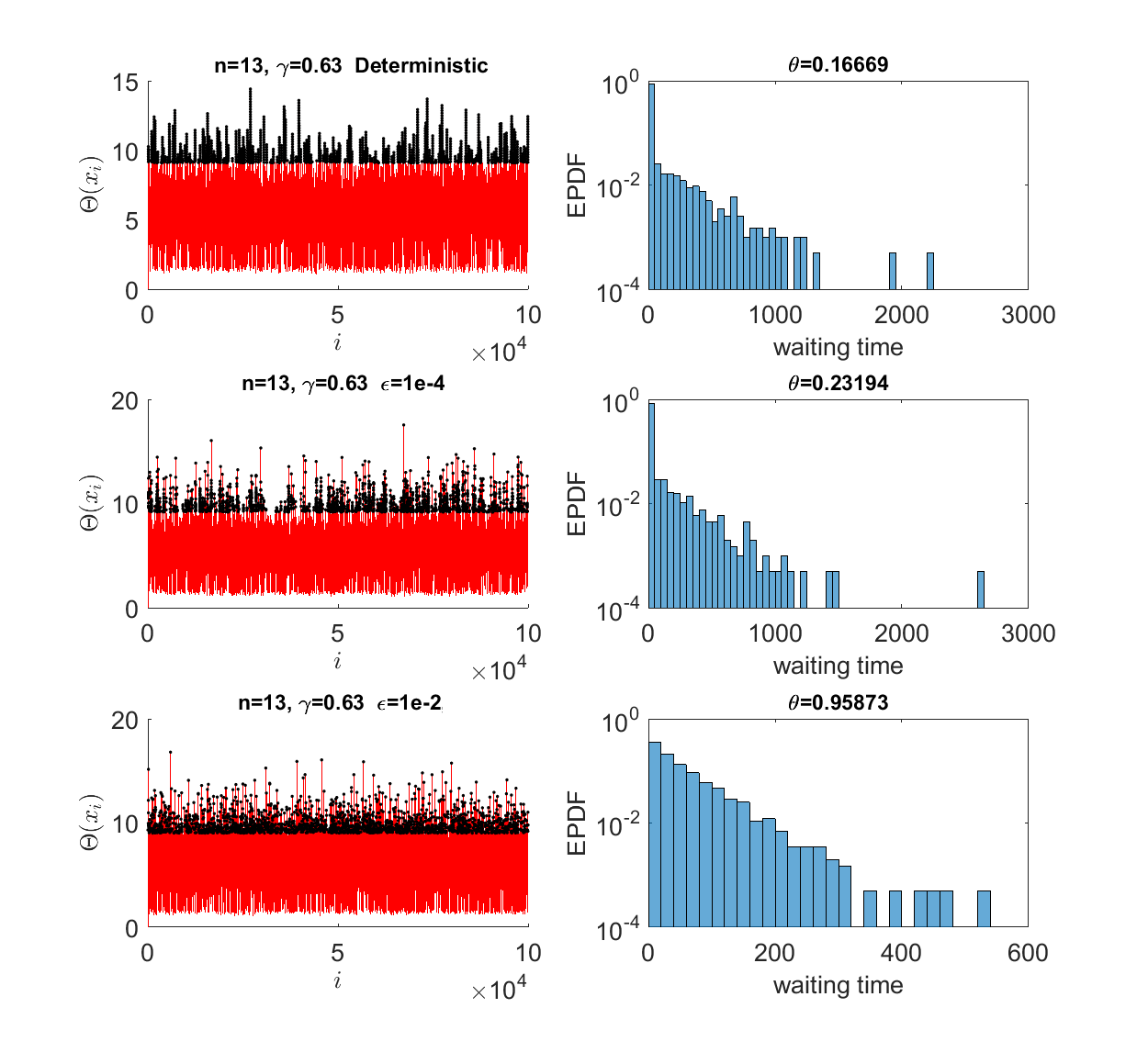}
\caption{Example of local case. Right: $\Theta$ time series (red) and exceedances (black). Left: empirical probability distribution (EPDF) of waiting time in the clusters. From top to bottom: Deterministic, additive noise with intensity $\epsilon=10^{-4}$, additive noise with intensity $\epsilon=10^{-2}$.}
\label{local_Poisson}
\end{figure}

\end{document}